\pdfoutput=1
 \documentclass[11pt]{article}
\usepackage{fullpage}
\usepackage{amssymb}
\usepackage{graphicx}
\usepackage{xcolor} % A package to add color.
\usepackage{tensor}
 \usepackage{fullpage} % Sets all margins to 1 in.  
\usepackage{amsmath}
\usepackage{amsthm}
\usepackage{verbatim}
\usepackage[hyperfootnotes=false]{hyperref}
\usepackage{enumitem}
\usepackage{mathrsfs}
\usepackage{bbm}
\usepackage{bm}
\usepackage{textcomp,gensymb}
\usepackage{tikz}
\usepackage{float}
\usetikzlibrary{hobby,calc,patterns,angles,quotes,arrows,decorations.markings}
\usepackage{caption}
\usepackage{float}

\makeatletter 

\tikzset{curlybrace/.style={rounded corners=2pt,line cap=round}}%  
\pgfkeys{% https://tex.stackexchange.com/a/45129/121799
/curlybrace/.cd,%
tip angle/.code     =  \def\cb@angle{#1},
/curlybrace/.unknown/.code ={\let\searchname=\pgfkeyscurrentname
                              \pgfkeysalso{\searchname/.try=#1,
                              /tikz/\searchname/.retry=#1}}}  
\def\curlybrace{\pgfutil@ifnextchar[{\curly@brace}{\curly@brace[]}}%

\def\curly@brace[#1]#2#3#4{% 
\pgfkeys{/curlybrace/.cd,
tip angle = 0.75}% 
\pgfqkeys{/curlybrace}{#1}% 
\ifnum 1>#4 \def\cbrd{0.05} \else \def\cbrd{0.075} \fi
\draw[/curlybrace/.cd,curlybrace,#1]  (#2:#4-\cbrd) -- (#2:#4) arc (#2:{(#2+#3)/2-\cb@angle}:#4) --({(#2+#3)/2}:#4+\cbrd) coordinate (curlybracetipn);
\draw[/curlybrace/.cd,curlybrace,#1] ({(#2+#3)/2}:#4+\cbrd) -- ({(#2+#3)/2+\cb@angle}:#4) arc ({(#2+#3)/2+\cb@angle} :#3:#4) --(#3:#4-\cbrd);
}
\makeatother

\setlist[enumerate]{leftmargin=1.5em}
\setlist[itemize]{leftmargin=1.5em}

\providecommand{\MR}{\relax\ifhmode\unskip\space\fi MR }
\providecommand{\href}[2]{#2}

\usepackage{seqsplit} 
\definecolor{green}{rgb}{0,0.8,0} % Redefines the color green.
\newcommand{\rom}[1]{\uppercase\expandafter{\romannumeral #1\relax}}

\newtheorem{maintheorem}{Theorem}

\newtheorem{theorem}{Theorem}[section]
\newtheorem{corollary}[theorem]{Corollary}
\newtheorem{conjecture}[theorem]{Conjecture}
\newtheorem{lemma}[theorem]{Lemma}
\newtheorem{proposition}[theorem]{Proposition}

\theoremstyle{definition}
\newtheorem{definition}[theorem]{Definition}

\theoremstyle{remark}
\newtheorem{remark}[theorem]{Remark}

\numberwithin{equation}{section}
\newcommand{\nnrm}[1]{{\vert\kern-0.25ex\vert\kern-0.25ex\vert #1 
    \vert\kern-0.25ex\vert\kern-0.25ex\vert}}

\newcommand{\supp}{{\mathrm{supp}}\,}

\newcommand{\nb}{\nabla}

\newcommand{\dlt}{\delta}

\newcommand{\veps}{\varepsilon}

\newcommand{\tht}{\theta}
\newcommand{\otht}{\overline{\theta}}

\newcommand{\omg}{\omega}
\newcommand{\oomg}{\overline{\omg}}
\newcommand{\ou}{\overline{u}}
\newcommand{\Omg}{\Omega}

\newcommand{\bfr}{{\bf r}}

\newcommand{\bbN}{\mathbb N}

\newcommand{\bbR}{\mathbb R}

\newcommand{\bbT}{\mathbb T}

\newcommand{\bbZ}{\mathbb Z}

\begin{document}

%\title{Velocity  blow-up in $C^1\cap H^2$ for the 2D Euler equations}
\title{Cusp Formation in Vortex Patches}

\author{
  Tarek M. Elgindi\thanks{Department of Mathematics, Duke University. E-mail: tarek.elgindi@duke.edu}
 \and
 Min Jun Jo\thanks{Department of Mathematics, Duke University. E-mail: minjun.jo@duke.edu}}
\date{\today}

 \renewcommand{\thefootnote}{\fnsymbol{footnote}}

\renewcommand{\thefootnote}{\arabic{footnote}}

%\date{\today}%
%\dedicatory{}%
%\commby{}%
% ----------------------------------------------------------------

\maketitle

% ----------------------------------------------------------------

\begin{abstract}
We prove instantaneous cusp formation for any initial vortex patch with acute corners. This was conjectured to occur in the numerical literature \cite{Carrillo,Danchin2}.
\end{abstract}

\setcounter{tocdepth}{1}
\tableofcontents

% ----------------------------------------------------------------

\section{Introduction}
We consider the vorticity form of the incompressible Euler equations in $\bbR^2$,
\begin{equation}\label{eq_Euler}
    \begin{gathered}
    \partial_t \omg + u \cdot \nb \omg = 0, \\ u=\nb^{\perp}\Delta^{-1} \omg,
\end{gathered}
\end{equation}
supplemented with the initial data $\omega(0,x)=\omega_0(x)$. The transport nature of \eqref{eq_Euler} gives rise to the classical global well-posedness theory developed by Yudovich \cite{Yudo} in $L^1 \cap L^\infty.$ This provides a rigorous framework where we can work with a particular family of solutions to \eqref{eq_Euler} called \emph{vortex patches.} Vortex patches are solutions of the form
\begin{equation}\label{def_patch}
    \omega(t,x) = \mathbf{1}_{\Omega(t)}(x),
\end{equation}
where $ \mathbf{1}_{\Omega(t)}$ denotes the characteristic function of a bounded set $\Omega(t)\subset \bbR^{2}$. More generally, any region of constant vorticity can be called vortex patch. 
% \color{red} Log-lipschitz velocity? \color{black}

\vspace{0.1in}

\noindent $\bullet$ \textbf{Boundary regularity}

\vspace{0.05in}

 The exact form \eqref{def_patch} reduces our target of analysis from the bulk $\Omega(t)$ to its boundary $\partial\Omega(t)$. This motivates study of boundary regularity, asking if $\partial \Omega(t)$ remains smooth when $\partial \Omega(0)$ is given to be smooth. Based on the numerical simulations \cite{Dritschel,Zabusky} in terms of the contour dynamics of $\partial \Omega$, Majda \cite{Majda} conjectured that $\partial \Omega(t)$ could evolve from certain smooth initial curve $\partial \Omega(0)$ into a non-smooth curve with cusps in finite time. After some debate in the numerical literature \cite{Buttke, Dritschel2}, the conjecture was negatively settled by Chemin \cite{Chemin} showing that the boundary regularity in $C^\infty$ or $C^{k,\alpha}$ with $k\geq 1$ and $0<\alpha<1$ is preserved for all time. Simplified proofs were provided later by Bertozzi-Constantin \cite{Constantin} and Serfati \cite{Serfati}. The key point in all the proofs is that $C^{1,\alpha}$ regularity of $\partial \Omega(t)$ gives the control of $\|\nb u(t)\|_{L^\infty}$.

 % \color{red}{geometric lemma} \color{black}

Such propagation of the boundary smoothness over time can be generalized to the notion of the propagation of any topological/geometric information that the boundary initially possesses. For instance, if the initial curve $\partial \Omega(0)$ bears a curvature singularity at a point, we may ask whether such an initial singularity would be preserved, lost, or transition to a more severe singularity. The non-locality encoded in the Biot-Savart law $u=\nb^{\perp} \Delta^{-1}\omega $ makes the question challenging to answer. A partial answer was given by Danchin \cite{Danchin1}, where it is shown that the patch boundary remains smooth along the image of the smooth part of $\partial\Omega(0)$. But the smoothness obtained in \cite{Danchin1} carries the possibility of non-uniformity up to singular points, while a corner-like singularity could give us, in principle, uniform smoothness up to the singular point. This leads to investigation of patches that initially have a corner singularity. The question is whether the corner structure remains or develops into another type of singularity.

\vspace{0.1in}

\noindent $\bullet$ \textbf{$\bm{m}$-fold symmetry and odd symmetry.}

\vspace{0.05in}

 What makes the behavior near a singular point difficult to study is that the Lipschitz norm of the velocity can be unbounded at the singularity. One way is to impose $m$-fold symmetry with $m\geq 3$ on the initial vorticity configuration. In \cite{EJ}, it is shown that $m$-fold symmetry with $m\geq 3$ imposed around a corner-like singularity makes the velocity gradient stay bounded, which leads to the global well-posedness for vortex patches with corners under the $m$-fold symmetry in a properly chosen functional setting.

 Aside from $m$-fold symmetry, oddity of vorticity configuration also has been extensively explored since the pioneering work \cite{KS} by Kiselev and Šverák on the double exponential growth of the vorticity gradient. Oddity does simplify the analysis of the Biot-Savart kernel as well, though it is not enough to bound the velocity gradient unlike the $m$-fold symmetry with $m\geq 3$. Indeed, the simplification via oddity occurs in a way that one can single-out the logarithmically divergent part, which was frequently used as the primary mechanism for illposedness in the critical Sobolev spaces for the Euler equations \cite{EJ2,JeongKim,JY,JoKim} and the generalized SQG equations \cite{CJJ,JeongKim2}. A prototype of the odd configuration is the so-called Bahouri-Chemin solution $\omega$ in $\bbT^{2}=[-1,1]^2$ that satisfies $\omega=1$ on $[0,1]^2$ under the odd-odd symmetry, see
 \hyperref[BCsol]{Figure 1.}
 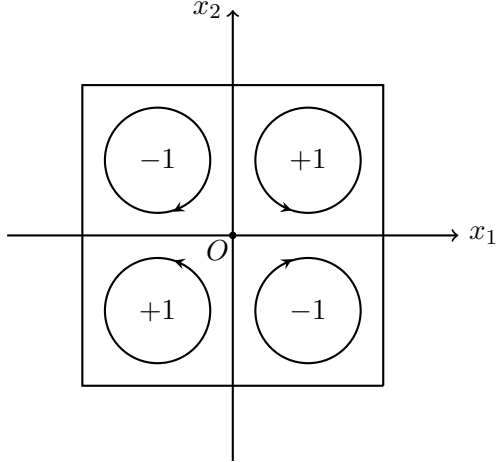
\begin{figure}[ht]
\begin{center}
\begin{tikzpicture}
    \coordinate (origo) at (0,0);
    \node (c) at (-0.2,-0.2) {$O$};
    \fill[black] (origo) circle (0.05);

    % draw axes
    \draw[thick,black,->] (-3,0) -- (origo) -- (3,0) node[black,right] {$x_1$};
    \draw[thick,black,->] (0,-3) -- (origo) -- (0,3) node (mary) [black,left] {$x_2$};

    % draw boundary
    \draw[thick,black,-] (2,-2) -- (2,2) -- (-2,2) -- (-2,-2) -- (2,-2);

    % sign
    \node (a) at (1,1) {$+1$};
    \node (b) at (-1,1) {$-1$};
    \node (c) at (-1,-1) {$+1$};
    \node (d) at (1,-1) {$-1$};
    % \node (a) at (2.2,-0.2) {$1$};
    % \node (a) at (-0.2,2.2) {$1$};
    % \node (a) at (-2.25,-0.2) {$-1$};
    % \node (a) at (-0.25,-2.2) {$-1$};

    \draw[thick,
        decoration={markings, mark=at position 0.7 with {\arrow[>=stealth]{>}}},
        postaction=decorate
        ]
        (1,1) circle (0.7);
            \draw[thick,
        decoration={markings, mark=at position 0.83 with {\arrow[>=stealth]{<
        }}},
        postaction=decorate
        ]
        (-1,1) circle (0.7);
            \draw[thick,
        decoration={markings, mark=at position 0.2 with {\arrow[>=stealth]{>}}},
        postaction=decorate
        ]
        (-1,-1) circle (0.7);
            \draw[thick,
        decoration={markings, mark=at position 0.33 with {\arrow[>=stealth]{<
        }}},
        postaction=decorate
        ]
        (1,-1) circle (0.7);
  \end{tikzpicture}
  \caption{Bahouri-Chemin solution on a periodic square}\label{BCsol}
  \end{center}
  \end{figure}
 One may notice that the Bahouri-Chemin solution is an equilibrium with its corners of $90^{\circ}$. Under the oddness condition, the special angles $0^{\circ}$ and $90^{\circ}$ like the ones found in the Bahouri-Chemin solution seem to be the only angles that propagate; in \cite{EJ}, it is also shown that if $\omega_0$ has its four acute corners meeting at the origin under the odd-odd symmetry, the angles at the corners immediately become $0^\circ$ (meaning cusp) or $90^{\circ}$. Whether it should be $0^\circ$ or $90^\circ$ can be determined by the specific initial configuration, for example, depending on the sign of $\omega_0$ on the first quadrant.
 
 % or the tangency of the corners to the $x_1$-axis.

 \vspace{0.1in}

\noindent $\bullet$ \textbf{Acute single corner}

\vspace{0.05in}

For single corners without any symmetry imposed, which correspond to the case $m=1$ with respect to the $m$-fold symmetry, there has been computational literature only \cite{Carrillo,Danchin2}. Based on their numerical simulations (see \hyperref[fig_cusp_CD]{Figure 2} and \hyperref[fig_tend_CD]{Figure 3} below), Cohen and Danchin \cite{Danchin2} conjectured that any single corner of acute angle would cusp immediately.

\begin{figure}[ht]\label{fig_cusp_CD}
    \centering
    \includegraphics[width=0.9\linewidth]{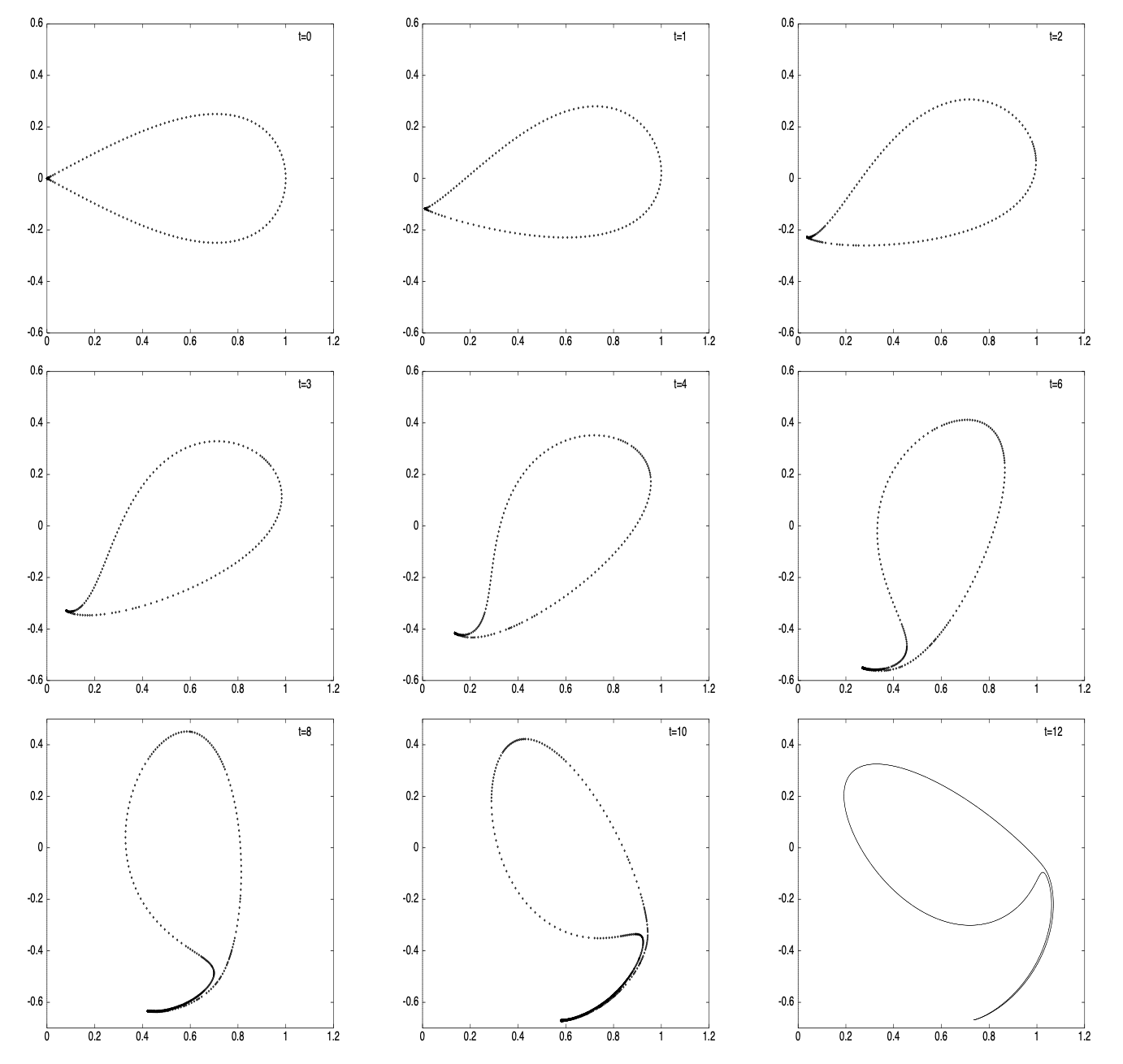}
    \caption{\cite{Danchin2}, p.497: Fig. 10, evolution of an initial acute-angle of a vortex patch.}
    
\end{figure}

\begin{figure}[ht]\label{fig_tend_CD}
    \centering
    \includegraphics[width=1\linewidth]{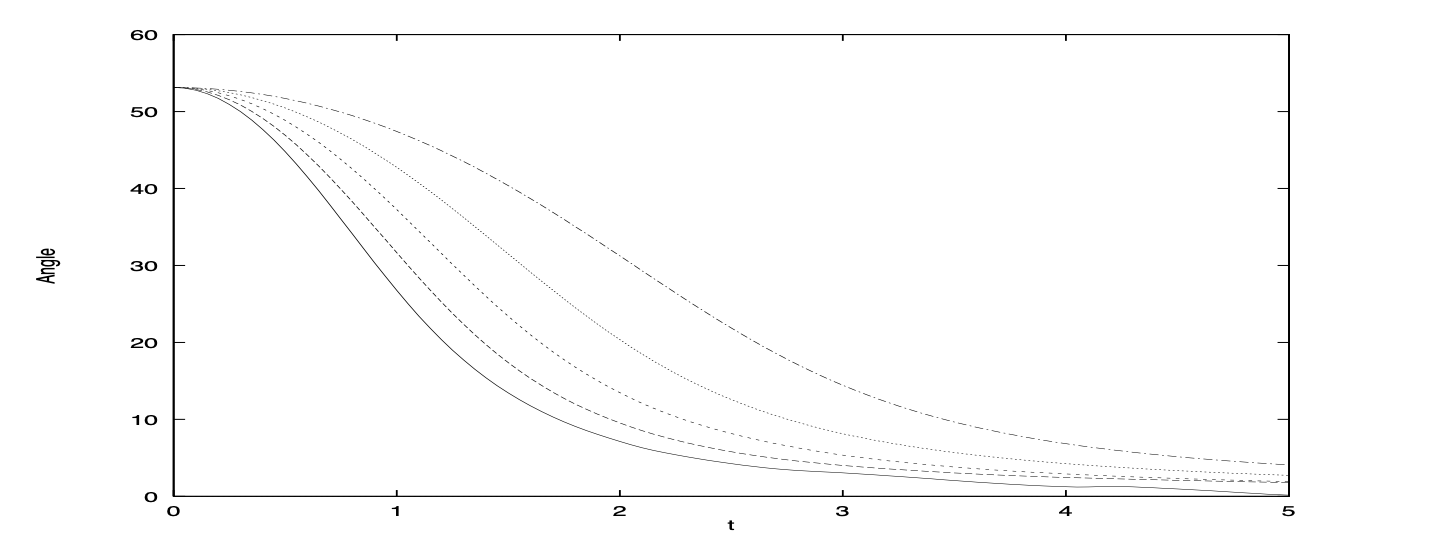}
    \caption{\cite{Danchin2}, p.494: Fig. 7(a), angle-in-time graph computed for different resolutions in the case of an initial patch with an acute corner. The resolution increases from top to bottom.}
   
\end{figure}
Moreover, from the proven fact \cite{Danchin3} cusps of order\footnote{The order of a cusp at $x=x_0$ is defined in \cite{Danchin2} by  $\sup\left\{\beta>0:\limsup_{x\to x_0}\frac{|(f_1-f_2)(x)|}{|x-x_0|^{\beta}}<\infty\right\}$, where $f_1$ and $f_2$ represent the graphs of the two sides of the cusping boundary near $x=x_0$. See Section 5.3 in \cite{Danchin2}. } one are preserved in time, the resulting order of the cusp evolving from an acute corner was conjectured in \cite{Danchin2} to be one with logarithmic sharpness. The full conjecture stated in \cite{Danchin2} is the following.
\begin{conjecture}[\cite{Danchin2}, p.496]\label{conj}
    Any acute single corner would immediately evolve into a cusp of order $1$ with logarithmic sharpness.  
\end{conjecture}
However, as briefly mentioned in \cite{Danchin2}, such logarithmic sharpness near a cusp singularity is numerically difficult to capture in general. On the analytical side of cusp formation, Elgindi-Jeong \cite[Chapter 6]{EJ} derived an asymptotic model that could be effective near the origin over a short time scale, using a finer velocity decomposition in polar coordinates together with a new spatially-weighted time variable $\tau=t|\ln r|$ where $r$ denotes the distance from the origin. Recently, Jeon-Jeong \cite{JJ} showed that for the generalized SQG equations, which is \eqref{eq_Euler} with $u$ recovered by the general Biot-Savart law $u=\nb^{\perp}\Delta^{-1+\alpha}\omega$ for $0\leq \alpha<1$, any initially corner-like patch immediately ``bends" either upward or downward. Especially for the Euler case $\alpha=0$, they proved that any single acute corner would immediately bend downward.

  \vspace{0.1in}

  \noindent $\bullet$ \textbf{Goal of this paper}

  \vspace{0.05in}

  This paper demonstrates Conjecture~\ref{conj}. The proof is two-fold: we analyze the \emph{effective} nonlinear ODE system that is formally derived under an asymptotic regime, and then we design a perturbative argument to deduce cusp formation for the original Euler equations \eqref{eq_Euler}.

\newpage
\subsection{Main theorem}

To motivate the main theorem, let us make some definitions. 

\begin{definition}\label{def_corner}
A simply connected set $\Omega$ is said to have a corner at $x=0$ of angle $0<\alpha <\pi$ if there exists a sector $\mathcal{S}$ of opening $\alpha$ so that 
\[\lim_{r\rightarrow 0}\frac{|B_r(0)\cap (\Omega\Delta\mathcal{S})|}{|B_r|}=0.\] A simply  connected set $\Omega$ is said to have a cusp at $x=0$ if $0\in\partial\Omega,$ but 
\[\lim_{r\rightarrow 0}\frac{|B_r(0)\cap \Omega|}{|B_r|}=0.\]
\end{definition}
Our main theorem is that if $\Omega_0$ has an acute corner at $x=0$ at time $t=0,$ the corner evolves instantaneously into a cusp. 
\begin{maintheorem}[Instantaneous cusp formation] \label{thm_main_arbitrary}
Assume $\omega_0=\chi_{\Omega_0}$ for some simply connected set $\Omega_0\subset \mathbb{R}^2$. Assume that $\Omega_0$ has an acute corner at $x=0$. Then the corresponding unique vortex patch solution $\omega(\cdot,t)=\chi_{\Omega(t)}$ develops a cusp instantaneously at the image of $x=0$ under the flow. In particular, 
\[\lim_{r\rightarrow 0}\frac{|B_r(\Phi_t(0))\cap \Omega(t)|}{|B_r|}=0,\] where we are evaluating at $t=t(r)\rightarrow 0$ as $r\rightarrow 0.$ Here, $\Phi_t(0)$ denotes the particle trajectory starting at $0.$
\end{maintheorem}

\begin{remark}
 See \hyperref[fig_initial]{Figure 4} and \hyperref[fig_cusp] {Figure 5} for illustrations of an initial patch and the corresponding cusp formation.\end{remark}

\subsubsection{Simplificiation via two-fold symmetry}

For simplicity, we prove instantaneous cusp formation under two-fold symmetry. More precisely, we assume that $\omega_0(r,\tht)=\omega_0(r,\pi+\tht)$ for any $\tht\in[0,2\pi]$ in polar coordinates. Indeed, it turns out that proving cusp formation under two-fold symmetry is enough to deduce the same conclusion for arbitrary single acute corners without any symmetry. One difference for minor modification is that the coordinates of the singular point move along the flow, which can be fixed by introducing the ``moving" frame as follows. Consider $\Phi_t(0)$ as the new origin where $\Phi_t(x)=(\Phi_t^1(x),\Phi_t^2(x))$ denotes the particle trajectory along the flow in time $t$ starting from $x$, i.e., the flow map associated with the velocity $u$. Then set up the moving polar coordinates as
\begin{equation*}
    \bfr = |x-\Phi_t(0)|, \quad \bm{\tht}=\arctan\left(\frac{x_2-\Phi_t^2(0)}{x_1-\Phi_t^1(0)}
    \right).
\end{equation*}
This enables us to mimic the perturbative argument in Section~\ref{sec_perturb} we use for the two-fold symmetry case in the usual polar coordinates.

Notice that by the translation invariance we can rewrite \eqref{eq_Euler} as
\begin{equation*}
    \partial_t \bm{\omega}(t,x) + (\bm{u}(t,x)-\bm{u}(t,0))\cdot \nb \bm{\omega}(t,x)=0
\end{equation*}
with $\omg(t,x)=\bm{\omg}(t,x-\Phi_t(0))$, meaning that the same approximate equation \eqref{eq_g_transport} will be derived via the velocity decomposition in Lemma~\ref{lem_key} for a single corner in the moving frame; see Remark~\ref{rmk_u}. The lack of the symmetry only changes the constants that appear during the further derivation, which does not affect the qualitative properties with respect to Section~\ref{sec_equivalence}-\ref{sec_decay}. Thus the same conclusions are drawn for the single corner case without any symmetry. %We refer to Section~\ref{sec_effective} for the full derivation of \eqref{eq_full}.

\subsubsection{Initial vorticity configuration}\label{sec_initial}
We consider the initial data of the form
\begin{equation*}
    \omega_0(x) = \mathbf{1}_{\Omega(0)}(x),
\end{equation*}
where the boundary of $\Omega(0)$ is smooth except at the origin $x=0$. We assume that $\Omega(0)$ is \textit{two-fold symmetric}, i.e., $\Omega(0)$ is invariant under rotation with angle $\pi$. If $\Omega(0)$ has an acute corner at $x=0$ (see Definition~\ref{def_corner}), there holds
\begin{equation}\label{eq_initial}
    \omega_0(r,\theta) = \mathbf{1}_{[A_0-B_0,A_0+B_0]}(\tht)+\mathbf{1}_{[\pi+A_0-B_0,\pi+A_0+B_0]}(\tht) + \Xi(r,\tht)
\end{equation}
for some numbers $A_0\in [0,2\pi]$, $B_0\in (0,\pi/4)$, and a function $\Xi:\bbR_{+}\times [0,2\pi]\to \{-1,0,1\}$ such that there exists a monotonic function $\kappa:\bbR_{+}\to\bbR_{+}$ satisfying
\begin{equation}\label{eq_initial_2}
\frac{1}{\pi r^2}\int_{0}^{r}s \int_{0}^{2\pi}|\Xi(s,\tht)|\,d\tht\,ds \leq \kappa(r) \quad \mbox{and} \quad \lim_{r\to 0^{+}}\kappa(r) = 0.
\end{equation}
For simplicity, we set $A_0=0$ and assume that $\omega_0$ is compactly supported in the region $0\leq r<1$.
See \hyperref[fig_initial]{Figure 4} for the simplest picture.

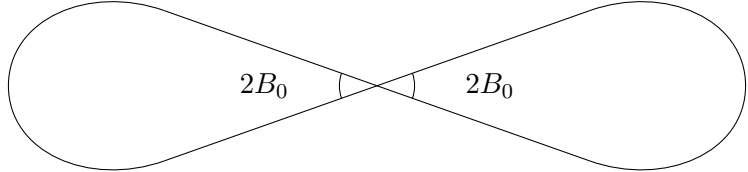
\begin{figure}[ht]
\begin{center}
\begin{tikzpicture}

\coordinate (origin) at (0,0);
\coordinate (mary) at ({sqrt(8)},1);
\coordinate (bob) at ({sqrt(8)},-1);
\coordinate (jason) at (-{sqrt(8)},1);
\coordinate (bell) at (-{sqrt(8)},-1);

    \draw[sharp corners=35pt]({sqrt(8)},-1)--(0,0)--({sqrt(8)},1);
    \draw[sharp corners=35pt](-{sqrt(8)},-1)--(0,0)--(-{sqrt(8)},1);
  \draw ({sqrt(8)},1) to[out angle=19.4712206, in angle=-19.4712206,
    curve through = { (4.9,0)}] ({sqrt(8)},-1);
   \pic [draw, -, "$2B_0$", angle eccentricity=3.0] {angle = bob--origin--mary};
   \draw (-{sqrt(8)},1) to[out angle=180-19.4712206, in angle=180+19.4712206,
    curve through = { (-4.9,0)}] (-{sqrt(8)},-1);
   \pic [draw, -, "$2B_0$", angle eccentricity=3.0] {angle = jason--origin--bell};
\end{tikzpicture}
\end{center}
\caption{A prototype of our initial patch $\Omega_0$ }\label{fig_initial}
\end{figure}

\subsubsection{Reformulation under two-fold symmetry}
% We call $J(t,r)$ the set of the angular support of $\omega(t)$ on $|x|=r$ defined by
% \begin{equation}\label{def_J}
% J(t,r)= \{\tht\in[0,2\pi]: \omega(t,r,\tht)=\mathbf{1}_{\Omega(t)}(r,\theta)=1 \}
% \end{equation}
% We denote by $\overline{\Tht}(t,r)$ the size of the angular support, where we set
% \begin{equation}\label{def_Tht}
%     \overline{\Tht}(t,r) = |J(t,r)|.
% \end{equation} 
% with $|J|$ denoting the measure of $J$. We also define the ``center'' angle $\Tht(t,r)$ by
% \begin{equation}\label{def_Tht_center}
% \Tht(t,r) = \frac{1}{|\int_{0}^{2\pi}\omg(t,r,\tht)\,d\tht|}\int_{0}^{2\pi}\tht\omg(t,r,\tht)\,d\tht  =  \frac{1}{|J(t,r)|}\int_{J(t,r)} \tht \, d\tht,
% \end{equation}
% which is just the angular center of mass over $J$.

We consider the average vorticity $G$ on balls centered at the origin:
\begin{equation*}
    G(t,r) = \frac{1}{|B_r|}\int_{B_r(0)}\omega(t,x)\,dx, \quad r>0.
\end{equation*}
 Notice that $G(t,r)$ coincides with $\frac{|B_r(0)\cap \Omega(t)|}{|B_r|}$ appearing in Definition~\ref{def_corner}. For the unique vortex patch $\Omega(t)$ evolving from $\Omega(0)$ that has a corner at the origin under two-fold symmetry, we establish the following theorem from which Theorem~\ref{thm_main_arbitrary} immediately follows.
\begin{theorem}\label{thm_main}
    Let $\omega(t,x)=\mathbf{1}_{\Omega(t)}(x)$ be the unique vortex patch solution to \eqref{eq_Euler} for the initial data $\omega_0(x)=\mathbf{1}_{\Omega(0)}(x)$ satisfying \eqref{eq_initial}-\eqref{eq_initial_2}. Then, there exists a continuous function $\eta:\bbR_{+}\to \bbR_{+}$ satisfying $$\lim_{r\to0^{+}}\eta(r)=0$$ such that there holds
    \begin{equation}\label{eq_thm}
      \lim_{r\to 0^{+}} G(\eta(r),r) = 0.
    \end{equation}
\end{theorem}
\begin{remark}
    See \hyperref[fig_cusp]{Figure 5} for a picture of cusp formation.
\end{remark}
     \begin{figure}[ht]
\begin{center}
\begin{tikzpicture}

\coordinate (origin) at (0,0);
\fill[black] (origin) circle (0.05);
\node (a) at (-0.1,-0.3) {$O$};
  \draw (0,0) to[out angle=-45, in angle=-45,
    curve through = { (2.2,-1) (3,1) (4,1.5) (-44.9:1) }] (0,0);

\draw (0,0) to[out angle=135, in angle=135,
    curve through = { (-2.2,1) (-3,-1) (-4,-1.5) (180-44.9:1) }] (0,0);
    \draw[->,>=stealth',thick] (150:1.2cm) (20:1.5) arc[radius=1.5, start angle=20, end angle=-35];
        \draw[->,>=stealth',thick] (150:1.2cm) (200:1.5) arc[radius=1.5, start angle=200, end angle=145];
    \node (a) at (1.3,0.9) {Angle jump};
    % \node (b) at (1.3,1.4) {potential};
\end{tikzpicture}
\end{center}
\caption{Instantaneous cusping near $r=0$ at $t\to 0^{+}$}\label{fig_cusp}
\end{figure}
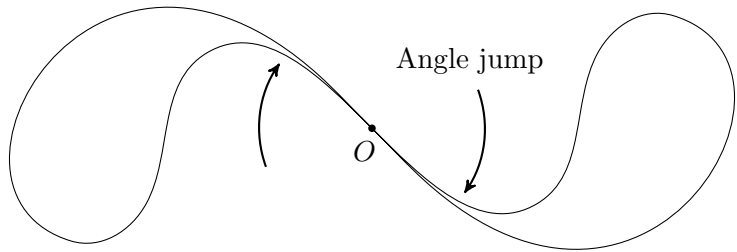

\begin{remark}
   For any $\kappa=\kappa(r)$ given in \eqref{eq_initial_2}, we establish the bound
    \begin{equation*}
        G(\eta(r),r) \lesssim \max\left\{\frac{1}{ \left(\ln |\ln r|\right)^{1+\delta}},\kappa^{1/2}\left(|\ln r|^{-1/2}\right)\right\}.
    \end{equation*}
    for some $\delta>0$. This implies the order-$1$ cusp with logarithmic sharpness, verifying Conjecture~\ref{conj} in the case that the initial corner is regular. Note that the reason for the double log is that the time scale is quite short. Indeed, the function $\eta$ is only slightly larger than $\frac{1}{|\ln r|}$, i.e.,  $\lim_{r\to 0}\left(\eta(r)|\ln r|\right) = \infty$. Given that the correct time scale for the problem is $\frac{\tau}{|\ln r|}$ (see \eqref{scale}), this should be seen as a long-time result with respect to $\tau$. It would be desirable to improve this to an all-time result.
\end{remark}

\subsection{Some related works}
We refer to \cite{EJ} as the primary source of  
the related references. For the reader's convenience, we briefly record a limited partial list of the previous works that could be relevant to this paper.

\begin{itemize}
\item General textbooks: Vortex dynamics has been a research target of huge interest as the simplified model of the fluid dynamics. See \cite{Lamb,Saffman} and references therein.

    \item Rotating patches: Under the $m$-fold symmetry with $m\geq 3$, it is known \cite{EJ} that the singular patches with their corners meeting at the origin rotate with some constant angular speed near the origin. Our work is about the cases $m=1$ and $m=2$. One may look at \cite{Burbea,DEJ, GPSY1,HMW,HM3,Park} and references therein.

    \item Critical phenomena for the 2D Euler equations: Largeness of velocity gradient is the key feature of instantaneous cusp formation, which often involves the ``corner-like" structure of the initial vorticity configuration that is not a patch but a smoothed-out version of patch. We refer to \cite{BL1,EJ2,EM,JeongKim,JY,JoKim,KL} and references therein.

    \item SQG patches: The $\alpha$-SQG equation generalizes \eqref{eq_Euler} in terms of the Biot-Savart law $u=\nb^{\perp}\Delta^{-1+\alpha}$ for which $\alpha=0$ corresponds to \eqref{eq_Euler} while $\alpha=1/2$ gives the classical SQG(surface quasi-geostrophic) equation. Due to its transport nature, the patch solutions naturally arise. See \cite{CCG1,JJ, KL2,KL3,KRYZ} and references therein.

\end{itemize}

\subsection{Heuristics of proof}
Let us now discuss some of the main ideas in the proof. The main ideas are essentially three. The first is that we derive and analyze an ODE system that describes the asymptotic behavior of a patch with a regular corner. This system is derived by keeping only the unbounded term of the velocity gradient, as was done previously in the work \cite{EJ}, and passing to a good coordinate system that reduces the asymptotic patch dynamics to an autonomous, nonlinear, fourth order ODE system. The ODE system models the evolution of the size and location of the corner. We completely characterize the long-time behavior of solutions to the ODE system by establishing various monotonicity properties of solutions. In particular, the angle of corner always closes in infinite (renormalized) time. Our analysis thus shows that the "leading order model" predicts cusp formation from a single acute corner. To lift the result from the ODE system, we devise a perturbative argument. Decomposing the solution $\omega$ as:
\[\omega=\omega_g+\bar\omega,\] where $\omega_g$ is given by the solution to the ODE system and $\bar\omega$ is simply $\omega-\omega_g$, we attempt to establish quantitative smallness of $\bar\omega$. Since we are working with patches, we first introduce the right quantity to measure $\bar\omega$:
\[F(t,r)=\frac{1}{|B_r|}\int_{B_r} |\bar\omega|.\] This quantity precisely measures the normalized area of the symmetric difference between the actual patch and the approximate one given by the ODE system.

It turns out that $F$ can be shown to satisfy an ODE of the type:
\[\partial_t F\leq C(1+t |\ln r|)+C\int_r^1\frac{F(s,t)}{s}ds,\]
\[\lim_{r\rightarrow 0}F(0,r)=0.\]
in the weak sense. Carrying this out requires a bit of work since we are working at relatively low regularity; however, the fact that the boundary of the approximate patch $\omega_g$ is rectifiable for all time allows us to close an approximation argument. After deriving the ODE satisfied by $F,$ we move to prove that it remains small for long (renormalized) time. Now, a-priori, we have that $|F|\leq 1.$ Plugging this into the ODE gives a bound of the type:
\begin{equation}\label{Firstbound}F\leq C t |\ln r|+F(0,r).\end{equation} Now, the time scale for $\omega_g$ is $\tau=t|\ln r|.$ Therefore, to see cusp formation, we have to prove smallness of $F$ for times 
\[t\gg\frac{1}{\ln r}.\] On such a long time scale, the bound \eqref{Firstbound} becomes useless. However, it turns out that the bound can be improved upon iteration. Indeed, every integration in time and integration with respect to $\frac{ds}{s}$ introduces another factor of $t|\ln r|$ and \emph{also an ever-improving constant}. In particular, after an $r$-dependent number of iterations, we are able to show that \[\lim_{r\rightarrow 0} F(t(r),r)=0,\] for some $t(r)=\frac{\xi(r)}{|\ln r|}$ with $\xi(r)\rightarrow\infty$ as $r\to 0$.

%Let $(A,B)$ be the solution pair to the effective ODE in \eqref{eq_full} and assume that we know $A(\tau)\to A_{\infty}\in [0,2\pi]$  and $B(\tau)\to 0$ as $\tau \to \infty.$ Then we may simplify \eqref{eq_full} near $\tau =\infty$ as
%\begin{equation*}
%    B''(\tau) = -\frac{B'(\tau)}{\tau} + \frac{(B'\tau))^2}{B(\tau)},
%\end{equation*}
%which admits its unique solution $B(\tau) = \tau^{-\frac{\tau_0 B'(\tau_0)}{B(\tau_0)}}$

\subsection{Outline of the paper}
Section~\ref{sec_effective} is devoted to the analysis of the effective system \eqref{eq_full}; we give a formal derivation of \eqref{eq_full} and prove that the effective patch angle $B$ decays to zero in the asymptotic regime. In Section~\ref{sec_perturb}, we design a perturbative argument based on the average of the perturbed vorticity near the singular point; the average quantity satisfies a recursive relation which eventually implies decay. We finish the proof of Theorem~\ref{thm_main} in Section~\ref{sec_main} by combining the results of the aforementioned sections.

\section{The Effective System}\label{sec_effective}

We start by introducing the ODE system in $\tau \in \bbR_{+}$ derived in \cite{EJ}:

\begin{equation}\label{eq_full}
    \begin{split}
        B''&=-\frac{B'}{\tau}+2\cot{(2B)}(B')^2-2\tan{(2B)}(A')^2, \\
        A''&= - \frac{A'}{\tau}-\frac{\sin{(4B)}}{\pi \tau}+2\cot{(2B)}A'B'-2\tan{(2B)}A'B',
    \end{split}
\end{equation}
supplemented by the initial data $0<B(0)=B_0 <\frac{\pi}{4},$ $B'(0)=0$, $A(0)=0,$ and $A'(0)=-\frac{\sin(4B_0)}{\pi}$. Here, $A$ represents the angle between the $x$-axis and the ray that bisects the vortex patch from the origin, and $B$ represents the half angle of the patch's corner. Such interpretation of $A$ and $B$ will be shown to be effective near $r=0$. The condition $B_0<\frac{\pi}{4}$ means that we treat patches whose initial angle is acute. Notice that $A(0)=0$ is chosen for simplicity, see Section~\ref{sec_initial}. Since $\tau$ will be interpreted as the spatially weighted time variable
\begin{equation}\label{scale}
    \tau = t|\ln r|
\end{equation}
for \eqref{eq_Euler} in the later sections, one may notice that investigating the asymptotic regime of $\tau\to \infty$ potentially corresponds to the regime $r\to 0^{+}$ with $ \frac{1}{|\ln r|} \ll t=t(r) \ll 1 $.

\subsection{Formal derivation of the effective system \texorpdfstring{\eqref{eq_full}}{Lg}}\label{sec_derivation}

We formally derive \eqref{eq_full} from the Euler equations \eqref{eq_Euler} in view of the patch solutions. One can find this derivation in Chapter 6 of \cite{EJ}; we repeat the same process for the reader's convenience. One may first recall the key lemma that was proved and used in \cite{EJ} as below.
\begin{lemma}[\cite{EJ}, Chapter 5]\label{lem_key}
    Assume that $\omega\in L_{c}^{\infty}(\bbR^2)$ and $\supp\omega$ is under the two-fold symmetry. Then, in polar coordinates, the velocity $u=\nb^{\perp}\Delta^{-1}\omega$ satisfies the estimate
    \begin{equation*}
        \left|u(r,\theta)-\frac{1}{2\pi}\begin{pmatrix}
            \cos\tht \\ -\sin\tht
        \end{pmatrix}rI^{s}(r) + \frac{1}{2\pi}\begin{pmatrix}
            \sin\tht \\ \cos\tht
        \end{pmatrix}rI^{c}(r)  \right|\leq Cr\|\omg\|_{L^\infty}
    \end{equation*}
    with some absolute constant $C>0$ independent on the size of the support of $\omega$, where
    \begin{equation}\label{def_int}
        I^{s}(r):= \int_{r}^{\infty}\int_{0}^{2\pi}\frac{\sin(2\tht)}{s}\omg(t,s,\tht)\,d\tht\,ds \quad \mbox{and}\quad I^{c}(r):=\int_{r}^{\infty}\int_{0}^{2\pi}\frac{\cos(2\tht)}{s}\omg(t,s,\tht)\,d\tht\,ds.
    \end{equation}
\end{lemma}
\begin{remark}\label{rmk_u}
    The only difference between the above lemma and the original one of \cite{EJ} (see also \cite{KS}) is that we use the fact $u(0,0)=0$ under the two-fold symmetry.
\end{remark}
\noindent The above lemma manifests that
\begin{equation*}
    u(r,\tht)= \frac{1}{2\pi}\begin{pmatrix}
            \cos\tht \\ -\sin\tht
        \end{pmatrix}rI^{s}(r) - \frac{1}{2\pi}\begin{pmatrix}
            \sin\tht \\ \cos\tht
        \end{pmatrix}rI^{c}(r)  + O(r).
\end{equation*}
Since we are interested in the regime $r\to 0^{+}$, the remainder $O(r)$ term is considered relatively small that we may drop it. This leads to the quasi-modified Euler equation (cf. \eqref{eq_modified})
\begin{equation*}
    \partial_t \omg + \left(\frac{1}{2\pi}\begin{pmatrix}
            \cos\tht \\ -\sin\tht
        \end{pmatrix}rI^{s}(r) - \frac{1}{2\pi}\begin{pmatrix}
            \sin\tht \\ \cos\tht
        \end{pmatrix}rI^{c}(r)\right)\cdot \nb \omg = 0.
\end{equation*}
Inspired by the fact that $I^{s}(r)$ and $I^{c}(r)$ are the integrals with respect to the measure $\frac{ds}{s}$, we write the vorticity as
\begin{equation*}
    \omega(t,r,\tht) = g(t|\ln r|,\tht) + \mbox{small remainder}
\end{equation*}
where the smallness of the remainder is assumed in the asymptotic regime $r\to 0^{+}$; we may drop the remainder term. Rewriting $I^{s}(r)$ and $I^{c}(r)$ as
\begin{equation*}
    \begin{split}
        I^{s}(r)&= \frac{1}{t}\int_{0}^{t|\ln r|}\int_{0}^{2\pi} \sin(2\tht')g(t|\ln r|,\tht')d\tht' d\left(t|\ln r|\right), \\
        I^{c}(r)&= \frac{1}{t}\int_{0}^{t|\ln r|}\int_{0}^{2\pi} \cos(2\tht')g(t|\ln r|,\tht')d\tht' d\left(t|\ln r|\right),
    \end{split}
\end{equation*}
we naturally introduce the new variable $\tau=t|\ln r|$ and arrive at
\begin{equation}\label{eq_mango}
\begin{split}
|\ln r|\partial_\tau g +\frac{r}{2\pi t} \bigg[ \begin{pmatrix}
            \cos\tht \\ -\sin\tht
        \end{pmatrix} J^{s}(\tau)  - \begin{pmatrix}
            \sin\tht \\ \cos\tht
        \end{pmatrix} J^{c}(\tau)\bigg] \cdot \bigg[ \frac{\partial_\tht g}{r}\begin{pmatrix}
            -\sin\tht \\ \cos\tht
        \end{pmatrix} - \frac{t\partial_\tau g}{r} \begin{pmatrix}
            \cos\tht \\ \sin\tht
        \end{pmatrix} \bigg] = 0
\end{split}
\end{equation}
where $J^{s}(\tau)$ and $J^{c}(\tau)$ are defined by
\begin{equation}\label{def_J_int}
\begin{split}
J^{s}(\tau) = \int_{0}^{\tau}\int_{0}^{2\pi}\sin(2\tht')g(\tau',\tht')\,d\tht'\,d\tau',  \quad
J^{c}(\tau)= \int_{0}^{\tau}\int_{0}^{2\pi}\cos(2\tht')g(\tau',\tht')\,d\tht'\,d\tau'.
\end{split}
\end{equation}
Dividing \eqref{eq_mango} by $|\ln r|$, we obtain
\begin{equation}\label{donut}
\partial_\tau g -\frac{1}{2\pi\tau}\bigg[\sin(2\tht)J^{s}(\tau)+\cos(2\tht)J^{c}(\tau)\bigg]\partial_\tht g = \frac{t}{2\pi\tau}\bigg[\cos(2\tht)J^{s}(\tau)-\sin(2\tht)J^{c}(\tau)\bigg]\partial_\tau g
\end{equation}
Since we are dealing with the ``instantaneous” phenomena, we assume $0<t \ll 1$ which formally justifies the negligibility of the right-hand side of \eqref{donut}. This gives us the following \emph{angular transport equation} for $g=g(\tau,\tht)$ as
\begin{equation}\label{eq_g_transport}
\partial_\tau g - \frac{1}{2\pi\tau} \bigg[ \sin(2\tht) J^{s}(\tau) + \cos(2\tht) J^{c}(\tau)  \bigg] \partial_\tht g=0.
\end{equation}
We now consider the initial data of the form
\begin{equation}
g(0,\tht) = \mathbf{1}_{[A_0-B_0,A_0+B_0]}(\tht) + \mathbf{1}_{[\pi+A_0-B_0,\pi+A_0+B_0]}(\tht)
\end{equation}
which coincides with \eqref{eq_initial} via the relationship $\omg_0(0,r,\tht) = g(0,\tht)$. This represents the initial corner (under the two-fold symmetry) whose angle is $2B_0$, centered at $A_0$. As we discussed in Section~\ref{sec_initial}, we impose $A_0=0$ without loss of generality. Due to the transport nature of \eqref{eq_g_transport}, the solution has the form
\begin{equation}\label{g_form_2}
g(\tau,\tht) = \mathbf{1}_{[(A-B)(\tau),(A+B)(\tau)]}(\tht) +  \mathbf{1}_{[\pi+(A-B)(\tau),\pi+(A+B)(\tau)]}(\tht).
\end{equation}
Evaluating \eqref{eq_g_transport} at $A+B$ and $A-B$ based on \eqref{g_form_2}, we obtain the coupled system
\begin{equation}\label{eq_ode_1}
\begin{split}
\tau B'(\tau) &= -\frac{1}{\pi}\bigg(\sin(2B)\cos(2A)\int_{0}^{\tau}\sin(2B)\sin(2A)\,d\tau' -\sin(2B)\sin(2A)\int_{0}^{\tau}\sin(2B)\cos(2A)\,d\tau'\bigg),\\
\tau A'(\tau) &= -\frac{1}{\pi}\bigg(\cos(2B)\sin(2A)\int_{0}^{\tau}\sin(2B)\sin(2A)\,d\tau'+\cos(2B)\cos(2A)\int_{0}^{\tau}\sin(2B)\cos(2A)\,d\tau'\bigg).
\end{split}
\end{equation}
Notice that the above integral system can be turned into a fourth-order ODE system as follows. We differentiate both sides of \eqref{eq_ode_1} and in the result of such differentiation, we replace some integral terms with $A'$ and $B'$ by using \eqref{eq_ode_1} again. This gives us
\begin{equation}\label{eq_second_order}
\tau   \begin{pmatrix} B'' \\ A'' \end{pmatrix} = - \begin{pmatrix} B'\\A' \end{pmatrix} - \frac{1}{\pi} \begin{pmatrix} 0 \\ \cos(2B)\sin(2B) \end{pmatrix} + \tau M' M  \begin{pmatrix} B' \\ A' \end{pmatrix}
\end{equation}
where $M$ is a matrix whose explicit form is
\begin{equation}\label{def_M}
M =  \begin{pmatrix} \sin(2B)\cos(2A) & -\sin(2B)\sin(2A) \\ \cos(2B)\sin(2A) & \cos(2B)\cos(2A)  \end{pmatrix}.
\end{equation}
From \eqref{def_M}, we know that
\begin{equation}
\begin{split}
   M' &= 2B'\begin{pmatrix} \cos(2B)\cos(2A) & -\cos(2B)\sin(2A) \\ -\sin(2B)\sin(2A) & -\sin(2B)\cos(2A)  \end{pmatrix}+ 2A'\begin{pmatrix} -\sin(2B)\sin(2A) & -\sin(2B)\cos(2A) \\ \cos(2B)\cos(2A) & -\cos(2B)\sin(2A)  \end{pmatrix} 
\end{split}
\end{equation}
and
\begin{equation}
M^{-1} = \frac{1}{\sin(2B)\cos(2B)} \begin{pmatrix} \cos(2B)\cos(2A) & \sin(2B)\sin(2A) \\ -\cos(2B)\sin(2A) & \sin(2B)\cos(2A)  \end{pmatrix}.
\end{equation}
Then we finally establish the desired ODE system \eqref{eq_full} for $A$ and $B$ from \eqref{eq_second_order}, which indeed will be proved to be \emph{effective} in the regime of $r\to 0^{+}$ with $0<t\ll 1$ via the perturbation argument presented in Section~\ref{sec_perturb}.

\subsection{Local well-posedness of \texorpdfstring{\eqref{eq_full}}{Lg}}
Since \eqref{eq_full} is a singular ODE, it may not be immediately obvious why local solutions exist starting from $\tau=0.$ The local existence is easily established using the contraction mapping theorem because of the specific choice of initial data. Indeed,
writing $g(\tau)=B'(\tau)$ and $f(\tau)=A'(\tau)+\frac{\sin(4B_0)}{\pi}$, the system \eqref{eq_full} reads
\begin{equation*}
 \begin{pmatrix}
     g \\ f
 \end{pmatrix}=
 \begin{pmatrix}
     \mathfrak{F_1}(g,f) \\ \mathfrak{F_2}(g,f)
 \end{pmatrix},
\end{equation*}
where the mapping $\mathfrak{F}:(g,f)\mapsto (\mathfrak{F}_1(g,f),\mathfrak{F}_2(g,f))$ is defined by
% \begin{equation*}
% \begin{split}
%      \mathfrak{F_1}(g,f)(\tau) &=  \frac{1}{\tau}\int_{0}^{\tau}2\tau' \left(\cot\left(2G(\tau')+2B_0\right)g(\tau')^2 - \tan\left(2G(\tau')+2B_0\right)\left(f(\tau')-c_1(B_0)\right)\right)^2\,d\tau', \\
%      \mathfrak{F_2}(g,f)(\tau) &=  \frac{1}{\tau}\int_{0}^{\tau}2\tau'\left(f(\tau')-c_1(B_0)\right)g(\tau') \left(\cot\left(2G(\tau')+2B_0\right)- \tan\left(2G(\tau')+2B_0\right)\right)\,d\tau' \\
%      &\qquad - \frac{1}{\tau}\int_{0}^{\tau}\frac{2}{\pi}\cos\left(2G(\tau)+4B_0\right)\sin(2G(\tau'))\,d\tau',
% \end{split}
% \end{equation*}
\begin{equation*}
\begin{split}
     \mathfrak{F_1}(g,f)(\tau) &=  \frac{1}{\tau}\int_{0}^{\tau}2\tau' \left(\cot\left(2G(\tau')+2B_0\right)g(\tau')^2 - \tan\left(2G(\tau')+2B_0\right)\left(f(\tau')-c_1(B_0)\right)\right)^2\,d\tau', \\
     \mathfrak{F_2}(g,f)(\tau) &=  \frac{1}{\tau}\int_{0}^{\tau}2\tau'\left(f(\tau')-c_1(B_0)\right)g(\tau') \left(\cot\left(2G(\tau')+2B_0\right)- \tan\left(2G(\tau')+2B_0\right)\right)\,d\tau' \\
     &\qquad - \frac{1}{\tau}\int_{0}^{\tau}\frac{1}{\pi}(\sin(4G(\tau')+4B_0)- \sin(4B_0))\,d\tau',
\end{split}
\end{equation*}
with the notations $ G(\tau')= \int_{0}^{\tau'}g(s)\,ds$ and  $ c_1(B_0)=  \frac{\sin(4B_0)}{\pi}.$ Then one can find sufficiently small $\veps>0$ depending on $B_0$ only such that the following two properties hold.
\begin{itemize}
    \item The mapping $\mathfrak{F}$ sends $X$ to $X$ where $$X=\left\{(g,f)\in C([0,\veps];\bbR)^2: \sup_{\tau\in[0,\veps]}|(g,f)(\tau)| \leq C(B_0)\veps \right\}.$$ Indeed, there exists a constant $C=C(B_0)>0$ such that
    \begin{equation*}
        \sup_{\tau\in[0,\veps]}|\mathfrak{F}(g,f)(\tau)| \leq C(B_0)\veps
    \end{equation*}
    for any $(g,f)\in X$.
    \item The mapping $\mathfrak{F}$ is a contraction. One can find a constant $C'=C'(B_0)>0$ such that
    \begin{equation*}
        \sup_{\tau\in[0,\veps]}|\mathfrak{F}(g_1,f_1)(\tau) - \mathfrak{F}(g_2,f_2)(\tau)| \leq C'(B_0) \veps \sup_{\tau\in[0,\veps]}|(g_1,f_1)(\tau)-(g_2,f_2)(\tau)| 
    \end{equation*}
   holds for any $(g_1,f_1)\in X$ and $(g_2,f_2)\in X$. Choosing $\veps < \frac{1}{2C'(B_0)}$, the mapping $\mathfrak{F}$ becomes a contraction on $X$.
\end{itemize}
By the contraction mapping principle, for any given $B_0\in (0,\pi/4)$, there exists a unique fixed point in $X$ for $\mathfrak{F}$ with sufficiently small $\veps>0$. This establishes the local well-posedness of \eqref{eq_full}. It is easy to deduce that $A$ and $B$ are smooth up to $\tau=0$, though we only use that $A,B\in C^2([0,\veps])$.

\subsection{Asymptotic equivalence of \texorpdfstring{$|A'|$}{Lg} and \texorpdfstring{$B$}{Lg} }\label{sec_equivalence}
In this subsection, we derive one specific property of the solution pair $(A,B)$ to the effective system \eqref{eq_full}: in the asymptotic regime $\tau\to \infty$, the quantities $|A'|$ and $B$ are comparable. We introduce a simple monotonicity lemma for $B$ and use the lemma to demonstrate the asymptotic equivalence between $|A'|$ and $B$. We then use this to deduce the global existence of solutions along with the fact that $B$ cannot become zero in finite time. 
\begin{lemma}\label{lem_B_0}
   So long as the solution exists, we have that non-increasing in $\tau$ and $0\leq B(\tau)< B_0$.
\end{lemma}
\begin{proof}
Observe that $B'=0$ implies $B''<0$ by the system \eqref{eq_full} so long as $B>0$. Since $B''(0)<0$ guarantees $B'<0$ near the initial time, a continuity argument says that $B'\leq 0$ so long as $B> 0$. %Suppose that $B=0$ and $B'<0$ at some $\tau>0$. This gives $B''=+\infty$, preventing $B$ from being negative.  
\end{proof}

% We use the taylor expansions for the trigonometric functions. For example,
% \begin{equation*}
%     \sin{x}=x+\frac{x^3}{3}+O(x^5), \,\, \cos{x}=1-\frac{x^2}{2}+O(x^4), \,\, \tan{x}=x + \frac{1}{3}x^3 + O(x^5), \,\, \cot{x}=\frac{1}{x}-\frac{x}{3}+O(x^3)
% \end{equation*}

\begin{proposition}\label{prop_A'}
    $|A'(\tau)|\approx B(\tau)$ so long as the solution exists.
\end{proposition}
\begin{proof}

To clarify the relation between $A'$ and $B$, we introduce the \textit{intermediate} quantity
\begin{equation}\label{def_Q}
    Q(\tau)=-\frac{\tau A'(\tau)}{B(\tau)}.
\end{equation}
Differentiating $Q$ in $\tau$ and using \eqref{eq_full}, we obtain
\begin{equation*}
\begin{split}
    Q'=-\frac{A'B+\tau A''B-\tau A'B'}{B^2} 
    % &=Q\left(2B'\left(\cot{(2B)}-\frac{1}{2B}-\tan(2B)\right)\right)+\frac{\sin(4B)}{\pi B} \\
    % &=Q\bigg(\log\left(\sin(2B)\right)-\log(2B)+\log(\cos(2B))\bigg)' +\frac{\sin(4B)}{\pi B} \\
    =Q\left(\log \left( \frac {\sin (4B)}{4B}\right)\right)' + \frac{\sin(4B)}{\pi B}.
\end{split}
\end{equation*}
% From the solution formula $Q(\tau)= \int_0^{\tau} e^{\int_{\tau'}^{\tau}f(s)\,ds}g(\tau')\,d\tau',$
Since $Q(0)=0$, we explicitly compute $Q$ as
\begin{equation}\label{eq_sol_Q}
    Q(\tau)=\frac{\tau \sin(4B(\tau))}{\pi B(\tau)}.
\end{equation}
Due to the monotonicity of $B$ from Lemma~\ref{lem_B_0}, there holds the trivial bound
\begin{equation}\label{B_0}
    \frac{\sin(4B_0)}{\pi B_0}\leq \frac{\sin(4B(\tau))}{\pi B(\tau)}\leq \frac{4}{\pi}
\end{equation}
for any $\tau>0$. Then \eqref{eq_sol_Q} and \eqref{B_0} imply that
\begin{equation*}
    Q(\tau)\approx \tau
\end{equation*}
uniformly in $\tau>0$. By the definition \eqref{def_Q}, we conclude that $|A'(\tau)|\approx B(\tau)$ as desired.
\end{proof}

\subsection{Global existence, the decay of \texorpdfstring{$B$}{Lg}, and the convergence of \texorpdfstring{$A$}{Lg}}\label{sec_decay}
% Inspired by the fact that the quantity $\frac{\tau_0 g'(\tau_0)}{g(\tau_0)}$ for any fixed $\tau_0>0$ precisely determines the order of the decay of the solution to the leading order ODE
% \begin{equation*}
%     g''(\tau)=-\frac{g'(\tau)}{\tau}+\frac{(g'(\tau))^2}{g(\tau)},
% \end{equation*}
% i.e., the solution $g$ has its explicit form as
% \begin{equation*}
%     g(\tau)= C(\tau_0)\tau^{\frac{\tau_0 g'(\tau_0)}{g(\tau_0)}} \quad \mbox{with} \quad C(\tau_0)=\frac{g(\tau_0)}{\tau_0^{\tau_0/g(\tau_0)}},
% \end{equation*}
We define the positive ``decay order" quantity $I$ by
\begin{equation}\label{def_decay_order}
    I(\tau)=-\frac{\tau B'(\tau)}{B(\tau)}.
\end{equation}
\begin{lemma}\label{lem_I_monotone}
   So long as the solution exists, $I(\tau)$ is non-decreasing in $\tau>0$.
\end{lemma}
\begin{proof}
Using \eqref{eq_full}, we observe that $I'(\tau)$ equals $2\tau$ times
\begin{equation}\label{eq_I'}
\begin{split}
   \frac{(B')^2}{B}\left(\frac{1}{2B}-\cot(2B)\right)+\frac{(A')^2}{B}\tan(2B).
\end{split}
\end{equation}
Since $(A')^2\approx B^2$ uniformly in $t>0$ by Proposition~\ref{prop_A'}, the second term is always non-negative. Therefore, it suffices to show that
\begin{equation*}
    \frac{1}{2B}-\cot(2B) \geq 0.
\end{equation*}
This is true simply because the function $h(x)=\frac{1}{x}-\cot(x)$ satisfies $h(x)\geq 0$ for any $x\in[0,\pi/2)$. For future use, we record the lower bound
\begin{equation}\label{I_lower} 
    I'(\tau) \gtrsim \tau B(\tau) \tan(2B(\tau))
\end{equation}
for any $\tau>0$, which uses the second term of \eqref{eq_I'} only.
\end{proof}
We can now deduce the global well-posedness of \eqref{eq_full} for any $0<B_0<\frac{\pi}{4}.$

\begin{corollary}
For any $0<B_0<\frac{\pi}{4},$ the local solution to \eqref{eq_full} can be extended to a global solution with $B$ decreasing and $|B(\tau)|\approx |A'(\tau)|,$ while $B$ cannot become zero in finite time. 
\end{corollary}
\begin{proof}
From Lemma \ref{lem_B_0} and Proposition \ref{prop_A'}, we have that $B'$ is uniformly bounded so long as the solution exists (i.e. so long as $B$ does not hit zero). From \eqref{eq_I'}, we have that 
\[\left(-\frac{\tau B'(\tau)}{B(\tau)}\right)'\leq C\tau( (B')^2+B^2).\]
It follows that 
\[\frac{-\tau B'(\tau)}{C\int_{0}^\tau \tau' ((B')^2+B^2) d\tau' }\leq B(\tau).\] Now, if $B$ becomes zero as $\tau\rightarrow\tau_*>0,$ we have that 
\[B'(\tau)\geq-\frac{B(\tau)}{C\tau_*},\] which leads to a contradiction. 
\end{proof}

\begin{lemma}\label{lem_I_1}
    There exists $T>0$ such that $I(T)>1$
\end{lemma}
\begin{proof}
    Suppose not for a contradiction. Then we have
\begin{equation*}
    (\log B)'(\tau) \geq -\frac{1}{\tau}
\end{equation*}
for all $\tau>0$. Integrating both sides from $1$ to $\tau>0$, we get $\log B(\tau) \geq \log (B(1) \tau^{-1})$. It leads to
\begin{equation}\label{ineq_contradiction}
    B(\tau) \geq \frac{B(1)}{\tau}.
\end{equation}
Combined with the fundamental theorem of calculus and the lower bound \eqref{I_lower}, the estimate \eqref{ineq_contradiction} further leads to
\begin{equation*}
    \begin{split}
        I(\tau)=-\frac{B'(1)}{B(1)}+\int_{1}^{\tau}I'(s)\,ds 
        \gtrsim \int_{1}^{\tau} s B(s) \tan(2B(s))\,ds \gtrsim \log{\tau}.
    \end{split}
\end{equation*}
We used $\tan(x) \geq x $ for $x\in[0,\pi/2)$. Therefore, $I(\tau)\to \infty$ as $\tau\to \infty$. A contradiction occurs.
\end{proof} 
\begin{theorem}\label{thm_B}
    For $\tau>T$, we have $B(\tau) \lesssim \tau^{-1-\delta}$ with $\delta>0$.
\end{theorem}
\begin{proof}
   Since $I(T)>1$ by Lemma~\ref{lem_I_1}, we can find $\delta>0$ such that $I(T)\geq 1+\delta>1$. Using the monotonicity of $I$ from Lemma~\ref{lem_I_monotone}, we further have $I(\tau)\geq 1+\delta$ for any $\tau\geq T$. This implies that
   \begin{equation*}
       (\log B)'(\tau) \leq - \frac{(1+\delta)}{\tau}.
   \end{equation*}
Integrating both sides from $T$ to $\tau$, we have $\log\left(B(\tau)\right) \leq \log \left( B(T) \tau^{-1-\delta}\right)$ and so
\begin{equation*}
    B(\tau) \leq B(T)\tau^{-1-\delta}
\end{equation*}
for any $\tau>T$.
\end{proof}

% \begin{corollary}\label{cor_A}
%     $A(\tau)$ converges to some number as $\tau\to \infty$.
% \end{corollary}
% \begin{proof}
% We use the formula $A'(\tau)= -\frac{Q(\tau)B(\tau)}{\tau}= -\frac{1}{\pi}\sin(4B(\tau))$ coming from \eqref{eq_sol_Q}. Then the initial conditions, $B(0)=B_0>0$ and $A(0)=0$, imply that $A(\tau)<0$ for any $\tau>0$ and $A(\tau)$ is monotonically decreasing. Moreover, $|A(\tau)|$ is uniformly bounded as
%     \begin{equation*}
% \begin{split}
% |A(\tau)|=\left|A(T)+\int_{T}^{\tau}A'(t)\,dt\right| \leq |A(T)|+ \frac{4}{\pi}\int_{T}^{\tau}B(t)\,dt \leq |A(T)| + \frac{4B_0}{\pi \delta T^{\delta}},
% \end{split}
%     \end{equation*}
%     thanks to Theorem~\ref{thm_B}. By the monotone convergence theorem, we conclude that $A(\tau)$ converges to its finite infimum as $\tau \to \infty$. Combined with the negativity of $A(\tau)$ for $\tau>0$, the convergence may represent the instantaneous ``turning" near the origin that occurs clock-wisely when it is effective to interpret $A$ as the center angle of the original patch, see \hyperref[fig_cusp]{Figure 3}. 
% \end{proof}

\begin{corollary}\label{cor_B'}
    For $\tau>T$, we have $|B'(\tau)| \lesssim \tau^{-2-\delta}$ with $\delta>0$.
\end{corollary}
\begin{proof}
We can get an explicit formula of $I(\tau)$ as we obtained one for $Q(\tau)$ in Proposition~\ref{prop_A'}. This gives us the relation between $B$ and $B'$ so that the decay of $B$ in Theorem~\ref{thm_B} yields the result.
%    Differentiating $I(t)$ and using the formula $A'(\tau)= -\frac{Q(\tau)B(\tau)}{\tau}= -\frac{1}{\pi}\sin(4B(\tau))$ that comes from \eqref{eq_sol_Q}, we achieve 
% \begin{equation*}
% \begin{split}
%     I'(\tau)
%     % &= -2t\frac{(B')^2}{B}\left(\cot(2B)-\frac{1}{2B}\right)+2t\frac{(A')^2}{B}\tan(2B) \\
%     % &= I \left(2B'\left(\cot(2B)-\frac{1}{2B}\right) \right) + 2t \left(-\frac{BQ}{t}\right)^2 \frac{\tan(2B)}{B} \\
%     % &=  I \left(\log\left(\frac{\sin(2B)}{2B}\right)\right)' + \frac{2B}{t} \tan(2B) \frac{4}{\pi} \frac{t\sin(4B)}{4B} \\
%     &= I \left(\log\left(\frac{\sin(2B(\tau))}{2B(\tau)}\right)\right)' +\frac{2}{\pi} \tan(2B(\tau))\sin(4B(\tau))
% \end{split}
% \end{equation*}
% From $I(0)=0$, we have the explicit solution formula
% \begin{equation*}
% \begin{split}
%     I(\tau)
%     % &= \frac{2}{\pi} \cdot \frac{\sin(2B(t))}{2B(t)} \int_{0}^{t} \frac{2B(\tau)}{\sin(2B(\tau))}\tan(2B(\tau))\sin(4B(\tau)) \,d\tau \\
%     &= \frac{4}{\pi} \cdot \frac{\sin(2B(\tau))}{2B(\tau)} \int_{0}^{\tau} 2B(\tau)\sin(2B(\tau)) \,d\tau.
% \end{split}
% \end{equation*}
% The definition $I(\tau)=-\frac{\tau B'(\tau)}{B(\tau)}$ further yields
% \begin{equation*}
% \begin{split}
%     |B'(t)|=  \frac{B(\tau)}{\tau}\frac{4}{\pi} \cdot \frac{\sin(2B(\tau))}{2B(\tau)} \int_{0}^{\tau} 2B(s)\sin(2B(s)) \,ds 
%     % &\leq \frac{16}{\pi}\frac{B_0} {\tau^{2+\delta}}  \int_{0}^{\tau} B^2(s)\,ds \\
%     % &\leq \frac{32B_0^3}{\pi} t^{-2-\delta} \\
%     \leq C\tau^{-2-\delta},
% \end{split}
% \end{equation*}
% once we apply Theorem~\ref{thm_B} for $t>T$ to get the desired decay rate. 
\end{proof}

\section{Perturbation Estimates}\label{sec_perturb}
In this section, we prove the validity of the effective system by proving that the actual Euler solution is just a perturbation of the solution given by the system. 

\subsection{Set-up of the perturbative regime}
As an approximation of the solution $\omega$ to \eqref{eq_Euler}, we choose $\omega_g$ of the form
\begin{equation*}
    \omega_g(t,r,\tht) = g(t|\ln r|,\tht) \quad \mbox{for}\,\,0<r\leq c_1 \leq 1
\end{equation*}
where $g:\bbR_{+}\times[0,2\pi]\to \bbR$ satisfies
\begin{equation*}
    g(\tau,\tht) = \mathbf{1}_{[(A-B)(\tau),(A+B)(\tau)]}(\tht) + \mathbf{1}_{[\pi+(A-B)(\tau),\pi+(A+B)(\tau)]}(\tht),
\end{equation*}
where $(A,B)$ is the solution pair to the previous effective system \eqref{eq_full}. In the next subsection, Section~\ref{subsec_deriv}, we will show that such an approximation $\omega_g$ solves the following \textit{modified Euler equations}
\begin{equation}\label{eq_modified}
\begin{gathered}
    \partial_t \omg_g +(1+h)u_g\cdot\nb \omg_g =0, \quad   u_g(t,r,\tht)= \frac{1}{2\pi}\left[\begin{pmatrix}\cos\tht \\
-\sin\tht        
    \end{pmatrix}rI_g^s(r) - \begin{pmatrix}
        \sin\tht \\ \cos\tht
    \end{pmatrix}r I_g^c(r) \right],
\end{gathered}
\end{equation}
for any $0<r\leq c_1\leq 1,$ where $I_g^s(r)$ and $I_g^c(r)$ are defined by
\begin{equation}\label{def_I}
    \begin{gathered}
   I_g^s(r)= \int_{r}^{\infty}\int_{0}^{2\pi}\frac{\sin(2\tht)}{s}g(t|\ln s|,\tht)\,d\tht\,ds,\quad I_g^c(r)= \int_{r}^{\infty}\int_{0}^{2\pi}\frac{\cos(2\tht)}{s}g(t|\ln s|,\tht)\,d\tht\,ds,
\end{gathered}
\end{equation}
and $h(t,x)$ is defined by
\begin{equation}\label{def_h}
    \begin{gathered}
 h(t,x) = \frac{(x\cdot u_g)t}{r^2|\ln r|-(x\cdot u_g)t}.
\end{gathered}
\end{equation}
Setting up the differences $\overline{\omg}=\omg-\omg_g$ and $\overline{u}=u-u_g$, we obtain the \textit{perturbative equation}

\begin{equation}\label{eq_perturb}
    \partial_t \oomg + u \cdot \nb  \oomg +  \ou\cdot \nb \omg_g - h u_g\cdot \nb \omg_g = 0 \quad \mbox{for}\,\,0<r\leq c_1 \leq 1.
\end{equation}
Lastly, we define the \textit{angle difference} $\otht$ by
\begin{equation}\label{def_otht}
    \otht(t,r)=\int_{0}^{2\pi}|\oomg(t,r,\tht)|\,d\tht 
\end{equation}
in the perturbative regime. See \hyperref[fig_angle]{Figure 6} for understanding of $\otht$.
\begin{figure}[ht]
\begin{center}
\begin{tikzpicture}

\coordinate (origin) at (0,0);
\coordinate (mary) at (65:4);
\coordinate (bob) at (25:4);
\coordinate (james) at (80:4);
\coordinate (clair) at (10:4);

    \draw[sharp corners=35pt](25:4)--(0,0)--(65:4);
   \pic [thick,draw, ->, "$\Theta$", angle eccentricity=2.0] {angle = clair--origin--james};
   % \node (b) at (4.7,4) {$\color{blue}\leq 2B +2\overline{\theta}$};
   \node (c) at (3.3,3.2) {$2B$};
   \node (c) at (4.4,1.3) {$\color{red} \leq \overline{\theta}$};
   \node (c) at (1.35,4.3) {$\color{red} \leq \overline{\theta}$};
       \draw[-,>=stealth',thick] (0:0cm) (4,0) arc[radius=4, start angle=0, end angle=90];
       \node (a) at (4,-0.25) {$|x|=r$};
    \draw[black,thick] (james) -- (0,0);
    \draw[black,thick] (clair) -- (0,0);
% \curlybrace[tip angle=-2,thick,color=blue]{80}{10}{5.2};
\curlybrace[tip angle=-2,thick,color=black]{65}{25}{4.17};
\curlybrace[tip angle=-2,thick,color=red]{25}{10}{4.17};
\curlybrace[tip angle=-2,thick,color=red]{80}{65}{4.17};
\end{tikzpicture}
\caption{A typical angle decomposition at $|x|=r$ for a single patch;}
\caption*{  $|\{\tht:|\oomg(r,\tht)|=1\}|\leq \otht$  and $|\{\tht:\oomg(r,\tht)=0\}| \leq 2B$ }
\end{center}
 \label{fig_angle}
\end{figure}
We record the below simple lemma for future use as a direct consequence of the definition \eqref{def_otht} with $\ou=u-u_g$.
\begin{lemma}\label{lem_ou}
    There holds
    \begin{equation}\label{est_lem}
        |\ou(r,\tht)| \leq Cr\int_{r}^{\infty}\frac{\otht(t,s)}{s}\,ds + Cr
    \end{equation}
\end{lemma}
\begin{proof}
    By Lemma~\ref{lem_key} and the definition of $u_g$ in \eqref{eq_modified}, we have
    \begin{equation*}
        u(r,\tht)-u_g(r,\tht) = \frac{r}{2\pi}\left[\begin{pmatrix}\cos\tht \\
-\sin\tht        
    \end{pmatrix}\left(I^s(r) -I_g^s(r)\right) - \begin{pmatrix}
        \sin\tht \\ \cos\tht
    \end{pmatrix} \left(I^c(r)-I_g^c(r)\right) \right] + O(r).
    \end{equation*}
 One may observe that
 \begin{equation*}
     |I^{j}(r)-I_g^{j}(r)| \leq \int_{r}^{1}\frac{\otht(t,s)}{s}\,ds
 \end{equation*}
% \begin{equation*}
%     \begin{split}
%         \cos\tht \Big(I^s(r) -I_g^s(r)\Big)-\sin\tht \Big(I^c(r)-I_g^c(r)\Big) & = \int_{r}^{\infty}\int_{0}^{2\pi}\frac{\sin(2\tht'-\tht)}{s}\Big(\omg(s,\tht')-g(t|\ln s|,\tht')\Big)\,d\tht'\,ds \\
%         &\leq \int_{r}^{\infty}\frac{\int_{0}^{2\pi}|\omg(s,\tht')-g(t|\ln s|,\tht')|\,d\tht'}{s} \,ds \\
%         &\leq \int_{r}^{\infty}\frac{\otht(t,s)}{s}\,ds
%     \end{split}
% \end{equation*}
for $j=s,c$. Then \eqref{est_lem} follows.
\end{proof}

\subsection{Derivation of the modified Euler equations}\label{subsec_deriv}
In this section, we rigorously derive the modified Euler equation \eqref{eq_modified} by reversing the formal derivation performed in Section~\ref{sec_derivation}. We consider the function $\omg_g$ of the form 
\begin{equation}\label{g_form}
    \omg_g(t,x)=\omg_g(t,r,\tht)=g(t|\ln r|,\tht) \,\,\, \mbox{for}\,\, 0<r\leq c_1 \leq 1,
\end{equation}
where $g$ is the solution to the angular transport equation \eqref{eq_g_transport}, which we record again as
\begin{equation}
    \begin{gathered}
        \partial_\tau g - \frac{1}{2\pi \tau}\bigg[\sin(2\theta)J^{s}(\tau) +\cos(2\theta)J^{c}(\tau)\bigg]\partial_\theta g=0.
    \end{gathered}
\end{equation}
where $J^s(\tau)$ and $J^{c}(\tau)$ are defined in \eqref{def_J_int}. Multiplying the above by $|\ln r|$ yields 
\begin{equation}\label{eq_g_2}
\begin{gathered}
    |\ln r|\partial_\tau g + \frac{r}{2\pi t}\bigg[\begin{pmatrix}\cos\theta \\
-\sin\theta        
    \end{pmatrix}J^{s}(\tau) - \begin{pmatrix}
        \sin\theta \\ \cos\theta
    \end{pmatrix} J^{c}(\tau)\bigg]
    \cdot\bigg[ \frac{\partial_\theta g}{r} \begin{pmatrix}
        -\sin\theta \\ \cos\theta
    \end{pmatrix}
    +\left(\frac{t\partial_\tau g}{r}-\frac{t\partial_\tau g}{r}\right)\begin{pmatrix}
        \cos\theta \\ \sin\theta
    \end{pmatrix}\bigg]=0.
\end{gathered}
\end{equation}
The rightmost two terms are from the simple trick of addition and subtraction. Then we observe that solving the above equation \eqref{eq_g_2} for $g=g(\tau,\theta)=g(t|\ln r|,\theta)$ is equivalent to solving for $\omega_g(t,r,\theta)=g(t|\ln r|,\theta)$ the following 
\begin{equation*}
\begin{gathered}
    \partial_t \omega_g+ u_g\cdot \nb \omega_g= - u_g \cdot \begin{pmatrix}
        x_1 \\ x_2
    \end{pmatrix} \frac{t}{r^2}\partial_\tau g, \quad
    u_g(r,\tht) = \frac{1}{2\pi}\left[\begin{pmatrix}\cos\theta \\
-\sin\theta        
    \end{pmatrix}rI_g^s(r) - \begin{pmatrix}
        \sin\theta \\ \cos\theta
    \end{pmatrix}r I_g^c(r) \right],
\end{gathered}
\end{equation*}
where $I_g^s(r)$ and $I_g^c(r)$ are defined in \eqref{def_I}. From the direct computation $\partial_\tau g(\tau,\theta) = \frac{1}{|\ln r|} \partial_t \omega_g(t,r,\theta)$, we see further that the above equation is  equivalent to
\begin{equation*}
   \left(1 - \frac{tx\cdot u_g}{r^2 |\ln r|} \right) \partial_t \omega_g+ u_g\cdot \nb \omega_g = 0.
\end{equation*}
Dividing the above by $\left(1 - \frac{tx\cdot u_g}{r^2 |\ln r|} \right)$, we get
\begin{equation*}
    \partial_t \omega_g+ \left(1+ \frac{(x\cdot u_g)t}{r^2|\ln r|-(x\cdot u_g)t}\right)u_g\cdot \nb \omega_g  = 0.
\end{equation*}
This leads to the desired modified Euler equations \eqref{eq_modified}-\eqref{def_h}. Our initial data should be matched with the condition
\begin{equation}
    g(0,\tht)=\mathbf{1}_{[-B_0,B_0]}(\tht) + \mathbf{1}_{[\pi-B_0,\pi+B_0]}(\tht)
\end{equation}
so that we can ensure the initial compatibility between $\omega$ and $\omega_g$ as
\begin{equation*}
\omega(0,r,\tht)-\Xi(r,\tht)=\omega_g(0,r,\tht)=g(0,\tht)= \mathbf{1}_{[-B_0,B_0]}(\tht) + \mathbf{1}_{[\pi-B_0,\pi+B_0]}(\tht)
\end{equation*}
for all $0< r \leq c_1$ with $c_1<1$ in view of \eqref{eq_initial}-\eqref{eq_initial_2} and \eqref{g_form}.

% \begin{theorem}[Perturbed cusp formation]\label{thm_cusp}
%      Let $\omega(t,x)=\mathbf{1}_{\Omega(t)}(x)$ be the unique vortex patch solution to \eqref{eq_Euler} for the initial data $\Omega(0)(x)=\mathbf{1}_{\Omega(0)}(x)$ given in Section~\ref{sec_initial}. For the corresponding quantity  $\otht$ defined in \eqref{def_otht}, there exists a continuous function  $\eta:\bbR_{+}\to \bbR_{+}$ satisfying $$\lim_{r\to0}\eta(r)=0$$ such that there holds
%     \begin{equation*}
%         \lim_{r0^{+}}F(t(r),r)=0
%     \end{equation*}
%     for any $t=t(r)$ satisfying $0<t(r)\leq \eta(r)$.
% \end{theorem}
% \begin{remark}
%     The function $\eta$ above is the same with the one in Theorem~\ref{thm_main}.
% \end{remark}

\subsection{Average estimates on the perturbed vorticity}

Let $\eta=\eta(r)$ denote the time scale for the patch solution to form a cusp at $|x|=r$.  Define the average of the perturbed vorticity $F=F(r)$ by
\begin{equation}\label{def_F_1}
    F(t,r)=\frac{1}{|B_r|}\int_{B_r}|\oomg(t,x)|\,dx.
\end{equation}
In polar coordinates, \eqref{def_otht} gives
\begin{equation}\label{def_F_2}
\begin{split}
    F(t,r)=\frac{1}{\pi r^2}\int_{0}^{r}\int_{0}^{2\pi}s|\oomg(t,s,\tht)|\,d\tht\,ds= \frac{1}{\pi r^2}\int_{0}^{r}s\otht(t,s)\,ds.
    \end{split}
\end{equation}
Let $F(t,r)=0$ for $r>1$, because we are looking at the region near the origin $r=0$.
\subsubsection{Smooth approximation}
Since we are working in the frame of weak solutions, we perform an approximation argument as follows. Let $\{\omg_0^{\varepsilon}\}_{\varepsilon>0}$ be a family of bump functions such that $\omg_0^{\veps} \to \omg_0$ as $\veps\to 0^{+}$ pointwisely and the following additionally holds:
\begin{itemize}
    \item  $\omg_0^\veps(x)=1$ if $ x\in\Omg(0)$, $\omg_0^\veps(x)=0$ if $d(x,\Omg(0))\geq \veps$, and $0\leq \omg_0(x) \leq 1$ if $d(x,\Omg(0))\leq \veps$,
    \item $|\nb \omg_0^\veps|_{L^\infty} \leq C\veps^{-1}$ for some absolute constant $C>0$,
    \item For each $x\in \bbR^2$, the sequence $\{\omg_0^\veps(x)\}$ is non-increasing as $\veps\to 0^{+}$,
    \item  $\supp \omg_0^\veps \subset \bbR^2$ is simply connected.
\end{itemize}
Consider another family of bump functions, $\{\omg_0^\dlt\}_{\dlt>0}$, satisfying the same properties described above. Then denote by $\omg^{\varepsilon}$ and $\omg_g^{\dlt}$ the corresponding smooth solutions to \eqref{eq_Euler} and \eqref{eq_modified}. Setting $\oomg^{\veps,\dlt}=\omg^{\veps}-\omg_g^{\dlt}$ and $\ou^{\veps,\dlt}=u^{\veps}-u_g^{\dlt}$, from \eqref{eq_Euler} and \eqref{eq_modified}, we deduce
\begin{equation}\label{eq_perturb_smooth}
    \partial _t \oomg^{\veps,\dlt} + u^\veps \cdot \nb \oomg^{\veps,\dlt} + \ou^{\veps,\delta}\cdot \nb \omg_g^\dlt - h^\dlt u_g^\dlt \cdot \nb \omg_g^\dlt = 0,
\end{equation}
which corresponds to \eqref{eq_perturb}. We also can define $\otht^{\veps,\dlt}$ and $F^{\veps,\dlt}$ accordingly to \eqref{def_otht} and \eqref{def_F_1}-\eqref{def_F_2}, simply replacing $\oomg$ by $\oomg^{\veps,\dlt}$. Note that $F(t,r)=\lim_{\veps,\dlt\to 0^+}F^{\veps,\dlt}(t,r)$ for any $t\geq0$ and $r>0$. The analogous domains are defined by $\Omg^{\veps}(t)=\operatorname{supp}\omg^{\veps}(t)$ and  $\Omg_g^{\dlt}(t)=\operatorname{supp}\omg_g^{\dlt}(t)$.

\subsubsection{Recursive estimates}
 For any $C_{\ast}>0$,
denote by $\Phi_t(r)$ the unique solution to the ODE
\begin{equation}\label{Phi}
\begin{cases}
    \partial_t \Phi(t,r) = \mathbf{v}(t,\Phi(t,r)), \\
    \Phi(0,r) = r,
\end{cases}
\end{equation}
where the velocity $\mathbf{v}=\mathbf{v}(t,r)$ is defined by
\begin{equation}\label{v_Phi}
    \mathbf{v}(t,r) = - C_{\ast} r \int_{r}^{1}\frac{B(t|\ln s|)}{s}\,ds.
\end{equation}
\begin{proposition}[Key Proposition]\label{prop_F} 
There exists an absolute constant $C_{\ast}>0$
 such that $\Phi$ determined via \eqref{Phi}-\eqref{v_Phi} satisfies 
 \begin{equation}
    F(t,\Phi_{t}(r)) \leq F(0,r)+ Ct+C\int_{0}^{t}\left(t'|\ln \Phi_{t'}(r)|+ \int_{\Phi_{t'}(r)}^{1}\frac{F(t',s)+F(t',\Phi_{t'}(r))}{s}\,ds\right)\,dt'
\end{equation}
for some absolute constant $C>0$.
\end{proposition}

\begin{proof}
 We test \eqref{eq_perturb_smooth} against $2\oomg^{\veps,\dlt}$. Integrating the result on the ball $B_r=B_r(0)$ and then dividing it by $|B_r|$, one obtains
    \begin{equation*}
    \begin{split}
          \partial _t F^{\veps,\dlt}(t,r) &= \frac{2}{|B_r|}\left(-\frac{1}{2} \int_{\partial B_r} |\oomg^{\veps,\dlt}|^2u^\veps\cdot n \,dS - \int_{B_r}\ou^{\veps,\dlt}\cdot\nb \omg_g^\dlt\oomg^{\veps,\dlt}\,dx + \int_{B_r}h^\dlt u_g^\dlt \cdot \nb\omg_g^\dlt \oomg^{\veps,\dlt}\,dx \right) \\
          &=: I + J + K,
    \end{split} 
    \end{equation*}
where integration by parts is used with $\operatorname{div}u^{\veps}=0$. The first term $I$ can be estimated as
\begin{equation*}
    \begin{split}
        I \leq \frac{1}{|B_r|} \int_{\partial B_r} |\oomg^{\veps,\dlt}|^2 \left(|\ou^{\veps,\dlt}|+|u_g^{\dlt}|\right)\,dS
    \end{split}
\end{equation*}
by the decomposition $u^\veps=\ou^{\veps,\dlt}+u_g^{\dlt}$.  Combining the fact $|\oomg^{\veps,\dlt}|^2\leq |\oomg^{\veps,\dlt}|$ with the definition of $\otht^{\veps,\dlt}$, we employ \eqref{lem_ou} to see 
\begin{equation*}
    I\leq C \otht^{\veps,\dlt}(t,r) + C \otht^{\veps,\dlt}(t,r) \int_{r}^{\infty}\frac{\otht^{\veps,\dlt}(t,s)}{s}\,ds + C\otht^{\veps,\dlt}(t,r) \int_{r}^{\infty} \frac{\tht_g^{\dlt}(t,s)}{s}\,ds.
\end{equation*}
Due to $\otht^{\veps,\dlt} \leq 2\pi$, the first term in the above is bounded by another constant $C$ with slight abuse of notation. To estimate the other terms, we observe the pointwise relation between $\otht^{\veps,\dlt}$ and $F^{\veps,\dlt}$:
\begin{equation}\label{relation_pointwise}
    \otht^{\veps,\dlt}(t,r) = 2\pi F^{\veps,\dlt}(t,r) +\pi r \partial_rF^{\veps,\dlt}(t,r).
\end{equation}
Since both $\omg^\veps$ and $\omg_g^{\dlt}$ are compactly supported in $B_1(0)$, \eqref{relation_pointwise} leads to
\begin{equation*}
    \begin{split}
        I \leq C + C\int_{r}^{1}\frac{F^{\veps,\dlt}(t,s)}{s}\,ds + C F^{\veps,\dlt}(t,r)\int_{r}^{1}\frac{\tht_g^{\dlt}(t,s)}{s}\,ds - \mathbf{v^{\dlt}}(t,r)\partial_r F^{\veps,\dlt}(t,r),
    \end{split}
\end{equation*}
where the velocity $\mathbf{v}^{\dlt}=\mathbf{v}^{\dlt}(t,r)$ is defined by
\begin{equation}\label{def_v}
    \mathbf{v}^{\dlt}(t,r) = - C r \int_{r}^{1}\frac{\tht_g^{\dlt}(t,s)}{s}\,ds.
\end{equation}
We estimate $J$ next. From $\oomg^{\veps,\dlt}=\omg^\veps-\omg_g^{\dlt}$, we decompose $J$ as
\begin{equation*}
    \begin{split}
        J &= -\frac{2}{|B_r|}\int_{B_r}\ou^{\veps,\dlt}\cdot \nb \omg_g^{\dlt}\omg^{\veps}\,dx +\frac{2}{|B_r|}\int_{B_r}\ou^{\veps,\dlt}\cdot \nb \omg_g^{\dlt}\omg_g^{\dlt}\,dx \\
        &=: J_1 + J_2.
    \end{split}
\end{equation*}
Integration by parts and \eqref{relation_pointwise} yield
\begin{equation*}
    \begin{split}
        J_2& = \frac{1}{|B_r|}\left(-\int_{B_r}\operatorname{div}\ou^{\veps,\dlt} |\omg_g^\dlt|^2\,dx + \int_{\partial B_r}|\omg_g^\dlt|^2\ou^{\veps,\dlt}\cdot n\,dS \right) \\
        &\leq C + C \tht_g^{\dlt}(t,r)\int_{r}^{\infty}\frac{\otht^{\veps,\dlt}(t,s)}{s}\,ds \\
        &\leq C + C \int_{r}^{1}\frac{F^{\veps,\dlt}(t,s)}{s}\,ds
    \end{split}
\end{equation*}
in a similar fashion to the previous estimate of $I$. We also see that $J_1$ satisfies
\begin{equation*}
    \begin{split}
        J_1 & = \frac{2}{|B_r|}\left(\int_{B_r \cap A^{\dlt}}\operatorname{div}(\ou^{\veps,\dlt}\omg^{\veps})\omg_g^{\dlt}\, - \int_{\partial (B_r\cap A^\dlt)} \omg_g^\dlt\omg^\veps\ou^{\veps,\dlt}\cdot n\,dS\right) \\
        &=  \frac{2}{|B_r|}\left(\int_{B_r \cap A^{\dlt}}\operatorname{div}(\ou^{\veps,\dlt}\omg^{\veps})\omg_g^{\dlt}\, - \int_{\partial B_r\cap A^\dlt} \omg_g^\dlt\omg^\veps\ou^{\veps,\dlt}\cdot n\,dS - \int_{B_r\cap \partial A^\dlt} \omg_g^\dlt\omg^\veps\ou^{\veps,\dlt}\cdot n\,dS\right) \\
        &=: J_{1,a} + J_{1,b} + J_{1,c},
    \end{split}
\end{equation*}
where $A^\delta:=\supp \nb \omg_g^{\dlt}$. Notice that the second equality holds due to the geometric feature of the surface integral oriented by the outward-pointing unit vector $n$. One can easily check that $J_{1,a}$ is simply bounded by $C(\veps)r^{-2}|A^{\dlt}|$ for some constant $C(\veps)$ depending only on $\veps>0$. For $J_{1,b}$, we compute
\begin{equation*}
    \begin{split}
        J_{1,b} &\leq \frac{C}{|B_r|}\int_{\partial B_r} |\ou^{\veps,\dlt}|\,dS \\
        &\leq C+C \int_{r}^{1}\frac{F^{\veps,\dlt}(t,s)}{s}\,ds 
    \end{split}
\end{equation*}
similarly to the previous estimates. It follows that
\begin{equation}\label{est_J}
    J \leq C+C \int_{r}^{1}\frac{F^{\veps,\dlt}(t,s)}{s}\,ds + C(\veps)r^{-2}|A^{\dlt}| + J_{1,c}.
\end{equation}
Since $K$ can be estimated analogously to the estimate of $J$, we simply record the result:
\begin{equation*}
    K \leq \left(C(\veps)r^{-2} + Ctr^{-1}|\ln r|\right)|A^{\dlt}| + K_{1,c},
\end{equation*}
where $K_{1,c}$ is defined by
\begin{equation*}
    K_{1,c} = \frac{2}{|B_r|}\int_{B_r\cap \partial A^{\dlt}}h^{\dlt}\omg_g^{\dlt}\omg^{\veps}u_g^{\dlt}\cdot n \,dS   .
\end{equation*}
The only difference is, compared to $J$, that $K$ involves $h^{\dlt}$ whose uniform-in-$\dlt$ pointwise estimate can be directly computed as
\begin{equation*}
    |h^{\dlt}(t,x)|\lesssim t \quad\mbox{and}\quad |\nb h^{\dlt}(t,x)| \lesssim \frac{t}{|x|}. 
\end{equation*}
Collecting all of the above estimates, we establish
\begin{equation}\label{before_flow}
\begin{split}
    \partial_t F^{\veps,\dlt}(t,r) + \mathbf{v}^{\dlt}(t,r)\partial_r F^{\veps,\dlt}(t,r) &\leq C + C\int_{r}^{1}\frac{F^{\veps,\dlt}(t,s)+F^{\veps,\dlt}(t,r)}{s}\,ds \\
    &\quad\quad\quad +\left(Ctr^{-1}|\ln r|+ C(\veps)r^{-2}\right)|A^{\dlt}| + J_{1,c}+K_{1,c}. 
\end{split}
\end{equation}
We postpone estimating $J_{1,c}$ and $K_{1,c}$ until we take the limit $\delta\to0^{+}$.

To perform further estimates along the flow advected by $\mathbf{v}^{\dlt}$, we define the corresponding flow map $\Phi^{\dlt}$ via the family of the initial value problems
\begin{equation}\label{flowmap}
\begin{cases}
    \partial_t \Phi^{\dlt}(t,r) = \mathbf{v}^{\dlt}(t,\Phi^{\dlt}(t,r)), \\
    \Phi^{\dlt}(0,r) = r.
\end{cases}
\end{equation}
\begin{lemma}\label{lem_log_lip}
    The velocity $\mathbf{v}^{\dlt}$ in \eqref{def_v} is log-Lipschitz continuous, independently of $\delta>0$.
\end{lemma}
\begin{remark}\label{rmk_zero} $\Phi^{\dlt}(t,0)=0$ for any $t\geq 0$ because
    $\mathbf{v}^{\dlt}(t,0) = 0$ for any $t\geq 0$. This also implies that if $r>0$, then $\Phi^{\dlt}(t,r)>0$ for all $t\geq 0$.
\end{remark}
\begin{remark}\label{rmk_loglip}
    For any fixed $r>0$, $\Phi^{\dlt}(t,r)$ is always decreasing in $t$.
\end{remark}
\begin{proof}[Proof of Lemma~\ref{lem_log_lip}]
Let $r<1$. Without loss of generality, for any $h>0$,
\begin{equation*}
    \begin{split}
        \left|\mathbf{v}^{\dlt}(r+h)-\mathbf{v}^{\dlt}(r) \right|&\leq  Ch\int_{r+h}^{1}\frac{\tht_g^{\dlt}(t,s)}{s}\,ds + Cr\int_{r}^{r+h}\frac{\tht_g^{\dlt}(t,s)}{s}\,ds \\
        &\leq Ch|\ln(r+h)| + C h\\
        &\leq C h |\ln h| + C h
    \end{split}
\end{equation*}
where we used $\frac{1}{r+h} \leq \frac{1}{h}$ in the last inequality. For $r\geq 1$, it is simply bounded by $Ch$.
    \end{proof}
\noindent A consequence of Lemma~\ref{lem_log_lip} is that the flow map ODE \eqref{flowmap} is globally well-posed. Furthermore, it ensures that we can rely on the well-known Yudovich estimates
\begin{equation}
\label{est_Yudovich}
\begin{split}
        r^{e^{ct}}\leq |\Phi_t^{\dlt}(r)-\Phi_t^{\dlt}(0)|&=|\Phi_t^{\dlt}(r)|=\Phi_t^{\dlt}(r)\leq r, \\
        r \leq |(\Phi_t^{\dlt})^{-1}(r)-(\Phi_t^{\dlt})^{-1}(0)|&=|(\Phi_t^{\dlt})^{-1}(r)|=(\Phi_t^{\dlt})^{-1}(r) \leq r^{e^{-ct}}
\end{split}
\end{equation}
where the trivial bounds come from Remark~\ref{rmk_zero}-\ref{rmk_loglip}. The constant $c$ in \eqref{est_Yudovich} is uniform in $\dlt>0$. Using the definition of $\Phi^{\dlt}$, one can obtain
\begin{equation*}
    \begin{split}
        F^{\veps,\dlt}(t,\Phi_t^\dlt(r))-F^{\veps,\dlt}(0,r) 
        &\leq Ct+ C\int_{0}^{t}\int_{\Phi_{t'}^{\dlt}(r)}^{1}\frac{F^{\veps,\dlt}(t',s)+F^{\veps,\dlt}(t',r)}{s}\,ds\,dt' \\
        &\quad +\int_{0}^{t}\left(Ct'\frac{|\ln \Phi_{t'}^{\dlt}(r)|}{\Phi_{t'}^{\dlt}(r)}+C(\veps)(\Phi_{t'}^{\dlt}(r))^{-2}\right)|A^{\dlt}|\,dt' + \int_{0}^{t} (J_{1,c}+K_{1,c})\,dt'.
    \end{split}
\end{equation*}
Now we take the limit $\dlt\to 0^{+}$. Every term in the first line converges by the dominated convergence theorem due to \eqref{est_Yudovich} and the uniform bound $$F^{\veps,\dlt}\leq 1.$$ For the second line, we know that
\begin{equation}\label{claim}
    \lim_{\dlt\to0^{+}}\sup_{t'\in [0,t]}|A^{\dlt}(t')|=0
\end{equation}
thanks to the Yudovich estimate associated with $(1+h^\dlt)u_g^\dlt$ and our choice of $\{\omg_0^{\dlt}\}_{\dlt>0}$. Then the first term in the second line vanishes as $\dlt\to 0^{+}$. We are left with the last term involving $J_{1,c}$ and $K_{1,c}$. See \hyperref[fig_approx]{Figure 7} below for general understanding before embarking on rigorous computation.

\begin{figure}[ht]
\begin{center}
\begin{tikzpicture}

\coordinate (origin) at (0,0);
\coordinate (mary) at (65:5);
\coordinate (bob) at (25:5);
\coordinate (james) at (80:5);
\coordinate (clair) at (10:5);

    \draw[black,thick,sharp corners=35pt](35:5)--(0.05,0.05)--(55:5);
    \draw[blue,thick,sharp corners=35pt](25:5)--(0,0)--(65:5);
    % \pic [thick,draw, ->, "$\Theta$", angle eccentricity=2.0] {angle = clair--origin--james};
    \node (b) at (4.3,3.95) {$\color{blue} \omega_g^{\delta} = 1$};
   % \node (a) at (3.3,3.2) {$2B$};
    \node (b) at (6.05,1.6) {$\color{blue}  0\leq \omega_g^{\delta} \leq 1$};
    \node (c) at (1.6,5.4) {$\color{blue} 0\leq \omega_g^{\delta} \leq 1$};
       \draw[-,>=stealth',thick] (0:0cm) (5,0) arc[radius=5, start angle=0, end angle=90];
       \draw[blue,->,dashed,>=stealth] (25:0cm) (4,1.9)  arc[radius=4, start angle=25, end angle=35];
       \draw[blue,->,dashed,>=stealth] (25:0cm) (1.9,4)  arc[radius=4, start angle=65, end angle=55];
       \node (d) at (5,-0.25) {$|x|=r$};
       \node (e) at (3,3) {$\Omega_g$};
       % \node (f) at (3.4,1.95) {$\delta\to 0$};
       % \node (g) at (5.4,2.2) {$ \color{blue} \partial \operatorname{supp} \nabla\omega_g^\delta$};
    \draw[blue,dotted] (james) -- (0,0);
    \draw[blue,dotted] (clair) -- (0,0);
% \curlybrace[tip angle=-2,thick,color=blue]{80}{10}{5.2};
\curlybrace[tip angle=-2,thick,color=blue]{65}{25}{5.17};
\curlybrace[tip angle=-2,dotted,color=blue]{25}{10}{5.17};
\curlybrace[tip angle=-2,dotted,color=blue]{80}{65}{5.17};
\end{tikzpicture}
\caption{A schematic picture of approximation as $\delta\to0^{+}$} 
\end{center}
 \label{fig_approx}
\end{figure}
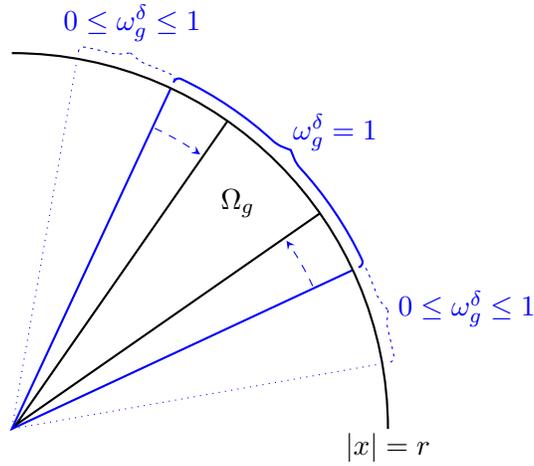

Recalling the definitions 
\begin{equation*}
    J_{1,c}=- \frac{2}{|B_r|}\int_{B_r\cap \partial A^\dlt} \omg_g^\dlt\omg^\veps\ou^{\veps,\dlt}\cdot n\,dS, \quad \mbox{and}\quad K_{1,c}=\frac{2}{|B_r|}\int_{B_r\cap \partial A^{\dlt}}h^{\dlt}\omg_g^{\dlt}\omg^{\veps}u_g^{\dlt}\cdot n \,dS,
\end{equation*}
we observe that due to the presence of $\omg_g^{\dlt}$ in both $J_{1,c}$ and $K_{1,c}$, the contribution coming from the outer part of $\partial A^{\dlt}$ among the two connected components disappears since $\omg_g^{\dlt}=0$ there. It suffices to check what occurs on the inner boundary during the limit $\dlt\to0^{+}$. We claim that
\begin{equation*}
\begin{split}
    \lim_{\dlt\to 0^{+}}\int_{0}^{t} (J_{1,c}+K_{1,c})\,dt' &= \int_{0}^{t} \lim_{\dlt\to 0^{+}}(J_{1,c}+K_{1,c})\,dt' \\
    &=\int_{0}^{t} \frac{1}{|B_{\Phi_{t'}(r)}|}\int_{B_{\Phi_{t'}(r)} \cap \partial \Omg_g}\left(\omg^{\veps}(u^{\veps}-u_g)-h\omg^{\veps}u_g\right)\cdot n\,dS,
\end{split}
\end{equation*}
where the first equality will be justified shortly after by the dominated convergence theorem and the second inequality will be the mere computation of their limits.

Consider a part of the initial boundary near the origin denoting a line segment: $$S=\{(s,c_1s)\in \bbR^2 : 0\leq s \leq c\}\subset \partial\Omg(0).$$ We can parametrize $S$ by $0\leq s\leq c$. The time evolution of $S$ can be expressed via $\phi_{g}^{\dlt}$ as
\begin{equation*}
   \phi_{g}^{\dlt}(S):= \{\phi_{g}^{\dlt}(t,s,c_1s):0\leq s \leq c\},
\end{equation*}
where $\phi_{g}^{\dlt}$ is defined by the ODE
\begin{equation*}
    \partial_t\phi_{g}^{\dlt}(t,x)=((1+h^{\dlt})u_g^{\dlt})(t,\phi_{g}^{\dlt}(t,x))
\end{equation*}
supplemented with the initial data $\phi_{g}^\dlt(0,x)=x$ for each $x\in \bbR^2$. Now we parametrize $\phi_{g}^{\dlt}(S)$ through the map $\gamma_{t,\dlt}:[0,c]\to \bbR^2$ defined by
\begin{equation*}
    \gamma_{t,\dlt} (s) = \phi_{g}^{\dlt}(t,s,c_1s).
\end{equation*}
Then we get, for instance,
\begin{equation*}
    \begin{split}
        J_{1,c}&= - \frac{2}{|B_{\Phi_{t'}(r)}|}\int_{B_{\Phi_{t'}(r)}\cap \partial A^\dlt} \omg_g^\dlt \omg^{\veps}\ou^{\veps,\dlt}\cdot n\,dS \\
        &=- \frac{2}{|B_{\Phi_{t'}(r)}|}\int_{0}^{c} \left(\omg_g^{\dlt}\omg^{\veps}\ou^{\veps,\dlt}\cdot n\right)(\gamma_{t',\dlt}(s))|\partial_s\gamma_{t',\dlt}(s)|\,ds,
    \end{split}
\end{equation*}
for some $c$ chosen to ensure $|\gamma_{t',\dlt}(c)|=\Phi_{t'}(r)$. The standard Yudovich estimate and the pointwise bound $|\nb u_g^{\dlt}(x)|\leq C|\ln|x||$ with $C>0$ independent of $\dlt>0$ give us that
\begin{equation*}
\begin{split}
\left(\omg_g^{\dlt}\omg^{\veps}\ou^{\veps,\dlt}\cdot n\right)(\gamma_{t,\dlt}(s))|\gamma_{t,\dlt}'(s)| &\leq C|\gamma_{t,\dlt}(s)|\big|\ln |\gamma_{t,\dlt}(s)|\big||\nb\phi_g^{\dlt}(\gamma_{t,\dlt}(s))| \\
    &\leq Cs^{1-2ct}|\ln s|e^{C|\ln s|t} \\
    &\leq Cs^{\frac{1}{2}}.
\end{split}
\end{equation*}
Therefore, the integrand is uniformly bounded in both $0<\dlt<1$ and $t'\in [0,t]$ with $t\le t_0$ for some sufficiently small absolute constant $t_0>0$ . We adopt the dominated convergence theorem to yield
\begin{equation*}
\begin{split}
    \lim_{\dlt\to0^{+}}\int_{0}^{t}J_{1,c}\,dt' &=- \int_{0}^{t}\frac{1}{|B_{\Phi_{t'}(r)}|}\int_{0}^{c}\left(\omg^{\veps}(u^{\veps}-u_g)\cdot n\right)(\gamma_t(s))|\partial_s \gamma_t(s)|\,ds\,dt' \\
    &= \int_{0}^{t} \frac{1}{|B_{\Phi_{t'}(r)}|}\int_{B_{\Phi_{t'}(r)}\cap \partial \Omg_g} \omg^{\veps}(u^\veps-u_g)\cdot n \,dS,
\end{split}
\end{equation*}
where $\gamma_t(s):=\gamma_{t,0}(s)=\phi_g(t,s,c_1s)$ and $\phi_g$ is the flow map associated with $(1+h)u_g$ so that $\gamma_t(s)$ represents the corresponding part of $\partial \Omg_g$ near the origin with the parameter $s$. The sign change in the second equality is due to the outward-pointing nature of $n$ in relation to the integral domain. Similarly, we can prove the claimed convergence for $K_{1,c}$ as $\dlt\to 0^{+}$. Once we complete taking the limit $\dlt\to 0^{+}$, we also send $\veps$ to zero from the right to obtain
\begin{equation*}
    F(t,\Phi_{t}(r)) \leq F(0,r)+ C\int_{0}^{t}\left(1+\int_{\Phi_{t'}(r)}^{1}\frac{F(t',s)+F(t',r)}{s}\,ds + \frac{1}{|B_{\Phi_{t'}(r)}|}\int_{B_{\Phi_{t'}(r)}\cap \partial \Omg_g}\big(|\ou|+|h||u_g|\big)\,dS\right)\,dt'.
\end{equation*}
To further estimate the last integral on the right-hand side,  we  parametrize $B_{\Phi_{t'}(r)} \cap \partial \Omg_g$ with the radial parameter $r'=|x|.$ The explicit parametrization is $$x=(x_1(r'),x_2(r'))=(r'\cos((A+(-1)^{j}B)(t|\ln r'|)), r'\sin((A+(-1)^{j}B)(t|\ln r'|)))$$ with $|x|=r'$ for $j=0,1$. Then we can see that the spatial integral involving $\ou$ is bounded by
\begin{equation*}
    \frac{2}{|B_{\Phi_{t'}(r)}|}\sum_{j=0,1}\left(\int_{0}^{\Phi_{t'}(r)}|\ou(r'\cos(A+(-1)^{j}B),r'\sin(A+(-1)^{j}B))|\sqrt{1+(t')^2|A'+(-1)^{j}B'|^2}\,dr'  \right),
\end{equation*}
where the dependence of $A$ and $B$ on $t'|\ln r'|$ is suppressed. This lets us arrive at
\begin{equation}\label{est_main_F}
    \begin{split}
       \frac{1}{|B_{\Phi_{t'}(r)}|}&\int_{\partial\Omega_g\cap B_{\Phi_{t'}(r)}}|\ou|\, dS \\
        &\leq \frac{C}{\Phi_{t'}(r)^2} \int_{0}^{\Phi_{t'}(r)} r'\int_{r'}^{1}\frac{F(s)}{s}\,ds \,dr' + C \\
        &= \frac{C}{ \Phi_{t'}(r)^2} \left( - \int_{0}^{\Phi_{t'}(r)} \frac{(r')^2}{2} \cdot \frac{-F(r')}{r'}\,dr' + \left[\frac{(r')^2}{2}\cdot \int_{r'}^{1}\frac{F(s)}{s}\,ds\right]_{r'=0}^{r'=\Phi_{t'}(r)}\right) + C \\
        &\leq C\int_{\Phi_{t'}(r)}^{1} \frac{F(s)}{s}\,ds + C.
    \end{split}
\end{equation}
 In the above calculation, we invoke the bound of $\ou$ similarly to the estimate of $I$ thanks to Lemma~\ref{lem_ou} and the uniform boundedness of $(t')^2(|A'|+|B'|)$ (see Proposition~\ref{prop_A'} and Corollary~\ref{cor_B'}) for the first inequality, and perform integration by parts for the equality in-between. The last inequality is just because $F\leq 1$. 
Lastly, the spatial integral involving $|h||u_g|$ is bounded by
\begin{equation*}
    \begin{split}
       \frac{Ct'}{\Phi_{t'}(r)^2}\int_{0}^{\Phi_{t'}(r)}\frac{r'}{|\ln r'|}&\left(\int_{r'}^{1}\frac{B(t'|\ln s|)}{s}\,ds\right)^2\,dr'\\  &\leq  \frac{Ct'}{\Phi_{t'}(r)^2}\int_{0}^{\Phi_{t'}(r)}r'|\ln r'|\,dr' \\
        &=  \frac{Ct'}{\Phi_{t'}(r)^2}\left(-\int_{0}^{\Phi_{t'}(r)}\frac{{r'}^2}{2}\left(-\frac{1}{r'}\right)\,dr' + \frac{\Phi_{t'}(r)^2}{2}|\ln \Phi_{t'}(r)|\right) \\ &\leq Ct'|\ln \Phi_{t'}(r)|,
    \end{split}
\end{equation*}
% \color{red} (taking time integration of the left-hand side gives a better decay $$\lesssim \frac{\ln(t|\ln r|)}{|\ln r|}$$ which goes to zero as $r\to 0$ for any time $t>0$)\color{black}
which can be verified precisely as done in the previous estimate with the aid of the estimate
\begin{equation*}
    h(t',x) \leq \frac{Ct'}{|\ln |x||}\int^{1}_{|x|}\frac{B(t'|\ln s|)}{s}\,ds
\end{equation*}
and the fact that $B(t'|\ln s|)\leq B_0<\pi/4$. Thus we reach
\begin{equation}
    \frac{1}{|B_{\Phi_{t'}(r)}|}\int_{B_{\Phi_{t'}(r)}\cap \partial \Omg_g}\big(|\ou|+|h||u_g|\big)\,dS \leq C+C\int_{\Phi_{t'}(r)}^{1} \frac{F(s)}{s}\,ds + Ct'|\ln\Phi_{t'}(r)|. 
\end{equation}
Collecting all of the above estimates, we finally establish that
\begin{equation*}
    F(t,\Phi_{t}(r)) \leq F(0,r)+Ct+C\int_{0}^{t}t'|\ln \Phi_{t'}(r)|\,dt'+ C\int_{0}^{t}\int_{\Phi_{t'}(r)}^{1}\frac{F(t',s)+F(t',r)}{s}\,ds\,dt'
\end{equation*}
as desired. 
\end{proof}

\subsection{Proof of Theorem~\ref{thm_main_arbitrary}}\label{sec_main}
\begin{proof}[Proof of Theorem~\ref{thm_main_arbitrary}]
We first treat the case $F(0,r)\equiv 0$ which represents the case of exact corners, and then prove Theorem~\ref{thm_main} in its full generality from which Theorem~\ref{thm_main_arbitrary} immediately follows.  

\noindent \textbf{Step 1. The simplest case of exact corners.}  From Proposition~\ref{prop_F}, we deduce
\begin{equation} 
    F(t,\Phi_t(r')) \leq Ct + C\int_{0}^{t}t'|\ln\Phi_{t'}(r')|\,dt' +C\int_{0}^{t}|\ln\Phi_{t'}(r')|\sup_{s\in[\Phi_{t'}(r'),1]}F(t',s)\,dt'.
\end{equation}
Writing $r' =\Phi_{t}^{-1}(r)$, we see that $\Phi_{t'}(r')=\Phi_{t'-t}(r) \geq r$ for all $t'\in[0,t]$ by \eqref{est_Yudovich}, which gives
\begin{equation}\label{eq_cusp_basic_2}
    F(t,r) \leq Ct + C\frac{t^2}{2}|\ln r| + C\int_{0}^{t}|\ln r| \sup_{s\in [r,1]}F(t',s)\,dt'
\end{equation}
The trivial fact $F\leq 1$ further leads to
\begin{equation}\label{ineq_cusp_result_1}
    F(t,r) \leq  F^{(0)}(t,r)
\end{equation}
where $F^{(0)}=F^{(0)}(t,r)$ is defined by
\begin{equation}\label{def_F0}
    F^{(0)}(t,r)= Ct + C\frac{t^2}{2}|\ln r|+Ct|\ln r|.
\end{equation} Here, the superscript $(0)$ denotes the number of the completed iterations whose procedure will be described below. A useful observation is that the upper bound $F^{(0)}(t,r)$ in \eqref{def_F0} is monotonically decreasing in $0<r\leq c_1$ (so it increases as $r\to 0$). This tells us that
 \begin{equation}\label{F0_prop_1}
 \begin{split}
        \sup_{s\in[r,1] } F^{(0)}(t,s)= F^{(0)}(t,r)
\end{split}
\end{equation}
Thanks to \eqref{F0_prop_1}, putting \eqref{ineq_cusp_result_1} back into \eqref{eq_cusp_basic_2} allows for the first \emph{iterated estimate} 
\begin{equation*}
    F(t,r) \leq Ct +C\frac{t^2}{2}|\ln r|+ C|\ln r|\int_{0}^{t}F^{(0)}(t',r)\,dt' =: F^{(1)}(t,r).
\end{equation*}
One may notice that such a process can be repeated in the exactly same way; for each integer $m\geq 1$, due to the recursive structure of \eqref{eq_cusp_basic_2} combined with the monotonicity property that corresponds to \eqref{F0_prop_1}, we obtain the $(m-1)$th iterated estimate
\begin{equation}\label{est_iterated}
   F(t,r) \leq Ct +C\frac{t^2}{2}|\ln r|+ C|\ln r|\int_{0}^{t}F^{(m-2)}(t',r)\,dt'=: F^{(m-1)}(t,r).
\end{equation}
Starting from \eqref{def_F0} within the recursive structure \eqref{est_iterated}, we obtain the explicit bound as
\begin{equation*}
    \begin{split}
 F(t,r) &\leq  \frac{1}{|\ln r|}\sum_{k=1}^{m} \frac{(Ct|\ln r|)^{k}}{k!} + \frac{1}{C|\ln r|}\sum_{k=1}^{m} \frac{(Ct|\ln r|)^{k+1}}{(k+1)!} +\frac{\left(Ct|\ln r|\right)^{m}}{m!} \\
    &\leq \frac{2}{|\ln r|}e^{Ct|\ln r|}+\frac{\left(Ct|\ln r|\right)^{m}}{m!}
\end{split}
\end{equation*}
where we use the Taylor expansion of $\operatorname{exp}(Ct|\ln r|).$ The Stirling formula estimate $m! \geq c_0 \left(\frac{m}{e}\right)^m$ further yields
\begin{equation}\label{F_bound_zero}
    \begin{split}
        F(t,r) \leq \frac{2}{|\ln r|}e^{Ct|\ln r|}+\frac{1}{c_0}\left(\frac{Ct|\ln r|}{m}\right)^{m},
    \end{split}
\end{equation}
which will be used below in deriving the desired conclusion.

\noindent \textbf{Step 2. The general case of arbitrary corner structure.} Now we treat the general case when $F(0,r) \leq \kappa(r)$ where $\kappa=\kappa(r)$ is given in \eqref{eq_initial_2}. Here, the only difference is that we need to estimate the \emph{descendants of} $\kappa(r)$, which newly arise during the iterations unlike the case $F(0,r)\equiv 0$. Define four operators $L$, $H$, $\tilde{L}$, and $\tilde{H}$ by
 \begin{equation}\label{def_LH}
 \begin{split}
          L(f)(t',r) &= \frac{1}{|\ln r|}\int_{r}^{1}\frac{f(t',s)}{s}\,ds,\quad H(f)(t',r) = \sup_{s\in[0,t]}f(t',\Phi_{s-t}(r)), \\
     &\tilde{L}(f)(t',r) = |\ln r| L(f)(t',r), \quad \mbox{and}\quad \tilde{H}(f)(t',r)=|\ln r| H(f)(t',r)
 \end{split}
 \end{equation}
for any function $f=f(t',r)$ that is continuous in $t'$ and $r$ with $t',r>0$. Once we recursively set
\begin{align*}
    F_2^{(m-1)}(t,r) &= \kappa(r) + C\int_{0}^{t} |\ln \Phi_{t'-t}(r)|\left(L(F_2^{(m-2)})(t',\Phi_{t'-t}(r)) + H(F_2^{(m-2)})(t',r)\right)\,dt', \\
    F_2^{(0)}(t,r) &= \kappa(r),
\end{align*}
we can establish the full bound
\begin{equation*}
    F(t,r) \leq F_1^{(m-1)}(t,r) + F_2^{(m-1)}(t,r), \quad \forall m\in \bbN
\end{equation*}
where $F_1^{(m-1)}$ is set by the previous iteration \eqref{est_iterated} together with  \eqref{def_F0}. Since we already know that $F_1^{(m-1)}$ enjoys the bound \eqref{F_bound_zero}, it suffices to estimate $F_2^{(m-1)}$. We first observe that, using $\Phi_{s-t}(r)\geq r$ for $s\in[0,t]$ in \eqref{est_Yudovich} and the definitions in \eqref{def_LH}, there holds
\begin{equation*}
    F_2^{(m-1)}(t,r) \leq \kappa(r) + C\int_{0}^{t} \left(\tilde{L}(F_2^{(m-2)})(t',r) + \tilde{H}(F_2^{(m-2)})(t',r)\right)\,dt'.
\end{equation*}
This provides us with
\begin{equation*}
\begin{split}
       F_2^{(m-1)}(t,r) &\leq \sum_{k=0}^{m-1} \frac{t^k|\ln r|^k}{k!}\cdot \frac{(\tilde{L}+\tilde{H})^{k}(\kappa)(r)}{|\ln r|^k},
\end{split}
\end{equation*}
thanks to the choice of $F_2^{(0)}(t,r)=\kappa(r)$. Moreover, since $\tilde{H}(\tilde{L}(\kappa))(r) \leq |\ln r|\tilde{L}(\kappa)(r)$ by \eqref{est_Yudovich} again, any term in the expansion of $(\tilde{L}+\tilde{H})^k$ can be controlled by a term of the form $|\ln r|^{k-i-j}\tilde{L}^i(\tilde{H}^j)$ with $i+j\leq k$
 for which one can prove the following. There holds
\begin{equation*}
     \frac{|\ln r|^{k-i-j}\tilde{L}^i(\tilde{H}^j(\kappa))(r)}{|\ln r|^k} \leq  \kappa(\dlt^{e^{-jct}}) + \frac{|\ln \dlt|}{|\ln r|}, \quad i+j\leq k,
\end{equation*}
so long as $0<r<\dlt$. Note that, in the following, we will have that $0<kt\ll 1.$ Using this bound along with our choice of $\dlt= \frac{1}{|\ln r|}$, we arrive at
\begin{equation}\label{F_2_bound}
    F_2^{(m-1)}(t,r) \leq \left(\kappa(|\ln r|^{-\frac{1}{2}})+\frac{\ln |\ln r|}{|\ln r|}\right) e^{Ct|\ln r|}.
\end{equation}

\noindent \textbf{Step 3. Choice of parameters.} \eqref{F_bound_zero} and \eqref{F_2_bound} give us the fully general bound
\begin{equation*}
    F(t,r) \leq \left(\frac{2}{|\ln r|}+\kappa(|\ln r|^{-\frac{1}{2}})+\frac{\ln |\ln r|}{|\ln r|}\right) e^{Ct|\ln r|}+ \frac{1}{c_0}\left(\frac{Ct|\ln r|}{m}\right)^{m}
\end{equation*}
for any $m\in \bbN$. Let $\xi:(0,c_1]\to \bbR_{+}$ satisfy
\begin{equation}\label{xi}
    \lim_{r\to 0^{+}} \xi(r) = \infty, \quad \mbox{and} \quad \lim_{r\to 0} \frac{\xi(r)}{|\ln r|} = 0.
\end{equation}
For any $0\leq t\leq \frac{\xi(r)}{C|\ln r|}$ with such $\xi$, one may have
\begin{equation}\label{eq_coffee}
    F(t,r) \leq   \left(\frac{2}{|\ln r|}+\kappa(|\ln r|^{-\frac{1}{2}})+\frac{\ln |\ln r|}{|\ln r|}\right) e^{\xi(r)} + \left(\frac{\xi(r)}{m}\right)^{m}.
\end{equation} 
Selecting large $m\in \bbZ $ that satisfies
\begin{equation}\label{est_m}
    e\xi(r) \leq m \leq  4\xi(r),
\end{equation}
the estimate \eqref{eq_coffee} can be followed by
\begin{equation}\label{ineq_final_form}
\begin{split}
    F(t,r) &\leq  \left(\frac{2}{|\ln r|}+\kappa(|\ln r|^{-\frac{1}{2}})+\frac{\ln |\ln r|}{|\ln r|}\right) e^{\xi(r)} + e^{-\xi(r)}.
\end{split}
\end{equation}
Note that \eqref{est_m} is valid whenever $r$ is small enough that $\xi(r)$ is greater than some absolute constant. It is left to specify $\xi$ to draw a conclusion. Pick $\xi$ to satisfy \eqref{xi} as, for example,
\begin{equation*}
    \xi(r) = \frac{1}{2}\min\left\{\ln |\ln r|, \ln\left(\frac{1}{\kappa\left(|\ln r|^{-1/2}\right)}\right) \right\}.
\end{equation*}
Then \eqref{ineq_final_form} enables us to estimate $F$ as its updated form
\begin{equation}\label{decay_F}
    F(t,r) \leq \frac{2\ln |\ln r|}{|\ln r|^{\frac{1}{2}}}+\kappa^{\frac{1}{2}}\left(|\ln r|^{-1/2}\right), 
\end{equation}
which goes to zero as $r\to 0^{+}$.

\noindent \textbf{Step 4. Conclusion for $G$.}
The result follows from the combination of Theorem~\ref{thm_B} and the estimate of $F$ obtained in \eqref{decay_F}. Observe first that
\begin{equation*}
    G(t,r) \leq F(t,r) + \frac{1}{|B_r|}\int_{0}^{r}s(4B(t|\ln s|))\,ds.
\end{equation*}
Define the time scale $\eta=\eta(r)$ by $\eta(r)= \frac{\xi(r)}{C|\ln r|}$.Using the bound $B(\tau)\lesssim (\tau)^{-1-\delta}$ for $\tau>T$ in Theorem~\ref{thm_B}, we have
\begin{equation*}
    \begin{split}
        \frac{1}{|B_r|}\int_{0}^{r}s(4B(\eta(r)|\ln s|))\,ds \lesssim \frac{1}{ r^2}\int_{0}^{r}\frac{s}{\left(\eta(r)|\ln s|\right)^{1+\delta}} \,ds \lesssim \frac{1}{(\ln |\ln r|)^{1+\delta}}
    \end{split}
\end{equation*}
for sufficiently small $r>0$. The estimate \eqref{decay_F} further yields
\begin{equation*}
    G(\eta(r),r) \lesssim \frac{\ln |\ln r|}{|\ln r|^{\frac{1}{2}}}+ \kappa^{\frac{1}{2}}\left(|\ln r|^{-1/2}\right) + \frac{1}{(\ln|\ln r|)^{1+\delta}},
\end{equation*}
which proves Theorem~\ref{thm_main} and so Theorem~\ref{thm_main_arbitrary} is established.
\end{proof}

\section{Some Questions}
While the main theorem establishes cusp formation in the sense mentioned in Theorem \ref{thm_main}, a number of interesting questions remain open, many of which could be attacked by refining some of the ideas presented above.

\subsection{Going to times of order 1}
In Theorem \ref{thm_main}, we prove cusp formation immediately, but we do not rule out that the corner could reconstitute itself immediately ``thereafter." It would be great to show that 
\[\lim_{r\rightarrow 0} G(t,r)=0,\] for all $t>0$, or for $t\in (0,\delta]$ for some fixed $\delta>0.$ This may require using the effective system in a more fundamental way. 

\subsection{Resolving the cusp and propagation of $C^1$ regularity}

It is reasonable to conjecture that the boundary of the patch near the origin is for all time given by the intersection of two $C^1$ curves (though, whose tangent vectors immediately become parallel at the origin). Proving this seems difficult, though it would imply that there is truly a cusp that forms for $t>0.$ A related question that we do not answer here is about the propagation of $C^1$ patches, though non-propagation of $C^2$ patches was proven in \cite{KL}. 

\subsection{Proving the angle jump}
As shown in \hyperref[fig_cusp]{Figure 5}, the ODE system for the size of the corner, $B(t)$, and its location, $A(t),$ gives that \[B(\tau)\rightarrow 0,\qquad A(\tau)\rightarrow A_\infty<0\] as $\tau\rightarrow\infty.$ In particular, the location of the corner/cusp jumps immediately at $r=0$ and $t>0$. This is also apparent from the numerics of \cite{Danchin2}. We can deduce this in an ``averaged sense," but it would be interesting to prove this in a pointwise sense. Note that $A_\infty<0$ indicates that the corner ``turns" in the opposite direction (clockwise) to the bulk.  

\subsection{Constructing patches with right angles}
As we have discussed, using odd symmetry (or working on domains\footnote{For example, it is not difficult to check that a half-disk patch set on the domain $B_1(0)$ rotates rigidly.}), it is easy to construct stationary or rotating patches with right corners. However, to our knowledge, it is not known whether there are rotating patches on $\mathbb{R}^2$ with right corners. These have been observed numerically in \cite{Carrillo}.

\section*{Acknowledgements}

The authors acknowledge partial support from NSF DMS-2043024. T.M.E. also acknowledges partial support from a Simons Fellowship and the Alfred P Sloan Foundation.

\section*{Competing Interests}
 The authors have no competing interests to declare that are relevant to the content of this article.

 \section*{Data Availability}
Data sharing not applicable to this article as no datasets were generated or analyzed during the current study.

%That the boundary remains locally $C^1$ on each side of the origin is not established here. Similarly, while the ODE system indicates that the 
%We prove cusp formation in an averaged sense, by using the averaged quantity $G(t,r)$. It would be great to prove that the boundary of the patch in a neighborhood of 0 is, for all time, the intersection of two $C^1$ curves with the same tangent vector at zero. It is not even clear if this is true.  

\bibliographystyle{amsplain}

\begin{thebibliography}{10}


\bibitem{Constantin} Bertozzi, A.L., Constantin, P. (1993). Global regularity for vortex patches. Comm. Math. Phys., 152(1), 19–28.


\bibitem{BL1}
Bourgain, J., Li, D. (2015). Strong ill-posedness of the incompressible Euler equation in borderline Sobolev spaces. Inventiones mathematicae, 201, 97-157.



\bibitem{Burbea}
    Burbea, J. (1982). Motions of vortex patches. Lett. Math. Phys., 6 no. 1, 1–16.


\bibitem{Buttke} Buttke, T.F. (1989). The observation of singularities in the boundary of patches of constant vorticity. Phys. Fluids A 1, 1283-1285.

\bibitem{Carrillo} Carrillo, J. A., Soler, S. (2000). On the evolution of an angle in a vortex patch. Journal of Nonlinear Science, 10, 23-47.


\bibitem{CCG1}
Castro, A., Córdoba, D., Gómez-Serrano, J. (2016). Existence and regularity of rotating global solutions for the generalized surface quasi-geostrophic equations. Duke Math. J., 165 no. 5, 935–984.



\bibitem{Chemin} Chemin, J. Y. (1993). Persistance de structures géométriques dans les fluides incompressibles bidimensionnels. In Annales scientifiques de l'Ecole normale supérieure, 26, No. 4, 517-542.


\bibitem{CJJ}
Choi, Y. P., Jung, J., Kim, J. (2024). On well/ill-posedness for the generalized surface quasi-geostrophic equations in H\"older spaces. arXiv preprint arXiv:2405.01245.



\bibitem{Danchin1} Danchin, R. (1997). Évolution temporelle d'une poche de tourbillon singulière. Communications in Partial Differential Equations, 22(5-6), 685-721.

\bibitem{Danchin2} Cohen, A., Danchin, R. (2000). Multiscale approximation of vortex patches. SIAM Journal on Applied Mathematics, 60(2), 477-502.

\bibitem{Danchin3} Danchin, R. (2000). Évolution d'une singularité de type cusp dans une poche de tourbillon. Revista matemática iberoamericana, 16(2), 281-329.


% \bibitem{DZ}
% Deem, G. S., Zabusky, N. J. (1978). Stationary ‘V-states,’interactions, recurrence and breaking, in “Solitons in Action”(K. Lonngren and A. Scott, Eds.). Academic Press, New York, 250, 277-293.



\bibitem{Dritschel}  Dritschel, D. (1985). Ph. Dissertation, Geophysical Fluid Dynamics Laboratory, Princeton.



\bibitem{Dritschel2} Dritschel, D. G.,  McIntyre, M. E. (1990). Does contour dynamics go singular?. Physics of Fluids A: Fluid Dynamics, 2(5), 748-753.

\bibitem{DEJ}
Drivas, T. D., Elgindi, T. M., and Jeong, I. J. (2024). Twisting in Hamiltonian flows and perfect fluids. Inventiones mathematicae,  238(1), 331-370.

\bibitem{EJ2}
    Elgindi, T. M., Jeong, I. J. (2017). Ill-posedness for the incompressible Euler equations in critical Sobolev spaces. Annals of PDE, 3, 1-19.

\bibitem{EJ} Elgindi, T., Jeong, I. J. (2023). On singular vortex patches, I: Well-posedness issues (Vol. 283, No. 1400). Memoirs of the American Mathematical Society.

\bibitem{EM}
Elgindi, T. M., Masmoudi, N. (2020). $L^\infty$ ill-posedness for a class of equations arising in hydrodynamics. Archive for Rational Mechanics and Analysis, 235(3), 1979-2025.






\bibitem{GPSY1}
Gómez-Serrano, J., Park, J., Shi, J., Yao, Y. (2021). Symmetry in stationary and uniformly rotating solutions of active scalar equations. Duke Mathematical Journal, 170(13), 2957-3038.



\bibitem{HMW}
 Hassainia, Z., Masmoudi, N., Wheeler, M. H. (2020). Global bifurcation of rotating vortex patches. Communications on Pure and Applied Mathematics, 73(9), 1933-1980.


\bibitem{HM3}
Hmidi, T., Mateu, J., Verdera, J. (2013). Boundary regularity of rotating vortex patches,
Arch. Ration. Mech. Anal., 209, no. 1, 171–208


\bibitem{JJ} Jeon, J., Jeong, I. J. (2023). On evolution of corner-like gSQG patches. Journal of Mathematical Fluid Mechanics, 25(2), 35.

\bibitem{JeongKim}
    Jeong, I. J., Kim, J. (2022). A simple ill-posedness proof for incompressible Euler equations in critical Sobolev spaces. Journal of Functional Analysis, 283(10), 109673.

    \bibitem{JeongKim2}
Jeong, I. J., Kim, J. (2024). Strong ill-posedness for SQG in critical Sobolev spaces. Analysis \& PDE, 17(1), 133-170.

\bibitem{JY}
    Jeong, I. J., Yoneda, T. (2021). Enstrophy dissipation and vortex thinning for the incompressible 2D Navier–Stokes equations. Nonlinearity, 34(4), 1837.

\bibitem{JoKim}
    Jo, M. J., Kim, J. (2022). Velocity blow-up in $ C^ 1\cap H^ 2$ for the 2D Euler equations. arXiv preprint arXiv:2211.13872.

\bibitem{Yudo}
Yudovich, V. I. (1963). Non-stationary flows of an ideal incompressible fluid. Zhurnal Vychislitel'noi Matematiki i Matematicheskoi Fiziki, 3(6), 1032-1066.

    

    \bibitem{KL}
    Kiselev, A., Luo, X. (2023). Illposedness of $C^2$ Vortex Patches. Archive for Rational Mechanics and Analysis, 247(3), 57. 


    \bibitem{KL2}
    Kiselev, A., Luo, X. (2023). On nonexistence of splash singularities for the $\alpha$-SQG patches. Journal of Nonlinear Science, 33(2), 37.


    \bibitem{KL3}
    Kiselev, A., Luo, X. (2023). The $\alpha$-SQG patch problem is illposed in $ C^{2,\beta} $ and $ W^{2, p} $. arXiv preprint arXiv:2306.04193.



\bibitem{KRYZ}
    Kiselev, A., Ryzhik, L., Yao, Y., Zlatoš, A. (2016). Finite time singularity for the modified SQG patch equation. Annals of mathematics, 909-948.


\bibitem{KS} Kiselev, A., Šverák, V. (2014). Small scale creation for solutions of the incompressible two-dimensional Euler equation. Annals of mathematics, 180(3), 1205-1220



\bibitem{Lamb}
Lamb, H. (1993) Hydrodynamics, 6th ed., Cambridge Mathematical Library, Cambridge University Press, Cambridge. With a foreword by R. A. Caflisch [Russel E. Caflisch]


\bibitem{Majda} Majda, A. (1986). Vorticity and the mathematical theory of incompressible fluid flow. Communications on Pure and Applied Mathematics, 39(S1), S187-S220.


% \bibitem{Overman}
% Overman II, E. A. (1986). Steady-state solutions of the Euler equations in two dimensions. II.
% Local analysis of limiting V-states. SIAM J. Appl. Math., 46, no. 5, 765–800


\bibitem{Park}
Park, J. (2022). Quantitative estimates for uniformly-rotating vortex patches. Advances in Mathematics, 411, 108779.

\bibitem{Saffman}
 Saffman, P. G. (1992). Vortex dynamics. Cambridge Monographs on Mechanics and Applied Mathematics, Cambridge University Press, New York.

% \bibitem{SS}
%  Saffman, P. G., Szeto, R. (1980). Equilibrium shapes of a pair of equal uniform vortices. Phys.
% Fluids 23, no. 12, 2339–2342.

\bibitem{Serfati}
Serfati, P. (1994) Une preuve directe d’existence globale des vortex patches 2D.
C. R. Acad. Sci., Sér. 1 Math. 318(6), 515–518.

% \bibitem{WOZ}
%       Wu, H. M., Overman II, E. A., Zabusky, N. J. (1984). Steady-state solutions of the Euler equa- tions in two dimensions: rotating and translating V-states with limiting cases. I. Numerical algorithms and results. J. Comput. Phys. 53, no. 1, 42–71.

% \bibitem{XJM}
%  Xue, B. B., Johnson, E. R., McDonald, N. R. (2017). New families of vortex patch equilibria for the two-dimensional Euler equations. Physics of Fluids, 29(12):123602.



\bibitem{Zabusky}  Zabusky, N., Hughes, M. H., Roberts, K. V. (1979). Contour dynamics for the Euler equations in two
dimensions. J. Comp. Phys., 96-106.


\end{thebibliography}

\end{document}